\documentclass{article}

\usepackage{hyperref}
\usepackage{amsmath}
\usepackage{amssymb}
\usepackage{amsthm}
\usepackage{tikz}
\usetikzlibrary{arrows.meta}

\newtheorem{lemma}{Lemma}
\newtheorem{theorem}{Theorem}
\newtheorem*{theorem*}{Theorem}

\newcommand{\inv}[1]{
	\!\;
	\overline{
		\!\!\: #1 \vphantom{!} \!\!\:
	}
	\;\!
}
\newcommand{\leqt}{ \trianglelefteq }
\newcommand{\up}[2]{ { ^{#1} \! {#2} } }

\DeclareMathOperator{\Image}{Im}
\DeclareMathOperator{\Aut}{Aut}
\DeclareMathOperator{\Hom}{Hom}

\DeclareMathOperator{\Mat}{M}
\DeclareMathOperator{\Out}{Out}
\DeclareMathOperator{\Jac}{J}
\newcommand{\sMat}[4]{
	\bigl(
		\begin{smallmatrix}
				#1 \mathstrut
			&
				#2 \mathstrut
			\\
				#3 \mathstrut
			&
				#4 \mathstrut
		\end{smallmatrix}
	\bigr)
}

\newcommand{\Sym}{ \mathrm{S} }
\newcommand{\Field}{ \mathbb{F} }
\newcommand{\Real}{ \mathbb{R} }
\newcommand{\Int}{ \mathbb{Z} }

\newcommand{\fppf}{ \mathrm{fppf} }
\newcommand{\ModSheaf}{ \mathop{ \mathbb{W} } }

\newcommand{\e}{ \mathrm{e} }
\newcommand{ \RootSys }[2]{ \mathsf{#1}_{#2} }
\newcommand{ \TitsIndex }[5]{
	\up{#2}{ \mathsf{#1}_{#3, #4}^{#5} }
}

\DeclareMathOperator{\Heis}{Heis}
\DeclareMathOperator{\Hyp}{H}
\newcommand{\OddFormPar}{ \mathcal{L} }
\newcommand{\hdisc}{ \mathop{ \mathrm{hdisc} } }

\DeclareMathOperator{\GenLin}{GL}
\DeclareMathOperator{\SpecLin}{SL}
\DeclareMathOperator{\Elem}{E}
\DeclareMathOperator{\KFunc}{K}

\newcommand{ \ProjGenLin }{ \mathbb{PGL} }

\DeclareMathOperator{\Orth}{O}
\DeclareMathOperator{\SpecOrth}{SO}
\DeclareMathOperator{\Refl}{Refl}
\newcommand{ \SpecOrthSch }{ \mathbb{SO} }
\newcommand{ \ProjSpecOrthSch }{ \mathbb{PSO} }

\DeclareMathOperator{\Unit}{U}
\DeclareMathOperator{\SpecUnit}{SU}
\DeclareMathOperator{\ElemUnit}{EU}
\DeclareMathOperator{\KUnit}{KU}
\newcommand{\UnitSch}{ \mathbb{U} }

\newcommand{\liealg}{ \mathfrak{g} }

\DeclareMathOperator{\ChevGrp}{G}

\title{
	Solvability of isotropic
	\( \KFunc_1 \)-functor \\
	over semilocal rings
}

\author{
	Egor Voronetsky%
	\thanks{
		This research is supported by the Russian Science Foundation grant 26-11-00085.
	} \\
	Saint Petersburg University, \\
	7/9 Universitetskaya nab., \\
	St. Petersburg, 199034 Russia
}

\begin{document}
\maketitle

\begin{abstract}
	We show that the
	\( \KFunc_1 \)-functor modeled on simple reductive groups over semilocal rings is solvable if the isotropic rank is at least
	\( 2 \) and that the Tits index is neither
	\( \TitsIndex{E}{2}{6}{2}{16''} \) nor
	\( \TitsIndex{E}{}{8}{2}{78} \). For these two Tits indices the result is already known, but assuming that the base ring contains a field. Our result implies that the elementary subgroup (or its derived subgroup) is the maximal perfect subgroup of the reductive group.
\end{abstract}

\section{Introduction}

Let
\( G \) be a reductive group scheme over a semilocal ring
\( K \) with connected spectrum. It is well known
\cite[Exp. XXVI]{red-grp-sch} that
\( G \) contains a minimal parabolic subgroup
\( P \leq G \) and all such parabolic subgroups are conjugate by the group
\( G(K) \). Fix two opposite parabolic subgroups
\( P^{+}, P^{-} \leq G \). The subgroup
\[
	\Elem_G(K)
	=
	\bigl\langle
		\mathrm{R}_{ \mathrm{u} }( P^{+} )(K),
		\mathrm{R}_{ \mathrm{u} }( P^{-} )(K)
	\bigr\rangle
	\leq
	G(K)
\]
generated by
\( K \)-points of the unipotent radicals of
\( P^{\pm} \) is called the elementary subgroup of
\( G(K) \), it is normal and independent of the choice of
\( P^{\pm} \). In the field case the elementary subgroup is also often denoted by
\( G(K)^{+} \).

The famous Kneser--Tits problem asks conditions for triviality of
\( \KFunc_1^G(K) = G(K) / \Elem_G(K) \) if
\( K \) is a field and
\( G \) is simple, simply connected, and isotropic. The positive answer is known in a number of cases, though in general there are counterexamples. See \cite{kne-tit}, especially theorem 6.1 for the list of Tits indices with trivial
\( \KFunc_1^G(K) \) over all fields, and separate papers
\cite{tit-wei-acp, tit-wei-tha} for the case
\( \TitsIndex{E}{}{8}{2}{78} \). Later Ph.~Gille and A.~Stavrova proved
\cite{k1-red-tri}[corollary 6.11] that
\( \KFunc_1^G(K) = 1 \) if
\( G \) is simple simply connected of isotropic rank at least
\( 2 \),
\( K \) contains a field, and one of the following situations holds.
\begin{itemize}
	
	\item
	The Tits index is
	\( \TitsIndex{A}{1}{n}{r}{ (d) } \), where
	\( r \geq 2 \),
	\( d \geq 1 \) is squarefree, and
	\( d (r + 1) = n + 1 \).
	
	\item
	The Tits index is
	\( \TitsIndex{A}{2}{n}{r}{ (d) } \), where
	\( r \geq 2 \),
	\( 1 \leq d \leq 3 \),
	\( d \mid n + 1 \),
	\( 2 r d \leq n + 1 \leq 2 r d + 3 \).
	
	\item
	The absolute root system is
	\( \RootSys{B}{\ell} \) or
	\( \RootSys{C}{\ell} \).
	
	\item
	The Tits index is
	\( \TitsIndex{D}{1}{n}{r}{ (d) } \), where
	\( r \geq 2 \),
	\( d \in \{ 1, 2, 4 \} \),
	\( \max(r d, 3) \leq n \),
	\( r d \neq n - 1 \), and
	\( n \) is even for
	\( d = 4 \).
	
	\item
	The Tits index is
	\( \TitsIndex{D}{2}{n}{r}{ (d) } \), where
	\( r \geq 2 \),
	\( d \in \{ 1, 2, 4 \} \),
	\( r d \leq n - 1 \), and
	\( n \) is even for
	\( d = 4 \).
	
	\item
	The Tits index is exceptional (including
	\( \TitsIndex{D}{3}{4}{2}{2} \) and
	\( \TitsIndex{D}{6}{4}{2}{2} \)) and distinct from
	\( \TitsIndex{E}{}{7}{2}{31} \).
	
\end{itemize}

In this paper we prove the following
\begin{theorem*}
	Let
	\( K \) be a semilocal commutative unital ring with connected spectrum and
	\( G \) be a simple reductive group scheme over
	\( K \). Suppose that the isotropic rank of
	\( G \) is at least
	\( 2 \) and in the cases of Tits indices
	\( \TitsIndex{E}{2}{6}{2}{16''} \) and
	\( \TitsIndex{E}{}{8}{2}{78} \) the base ring
	\( K \) contains a field. Then the group
	\( \KFunc_1^G(K) \) is solvable.
\end{theorem*}

Of course, the two exceptional cases are covered by
\cite{k1-red-tri}. For other cases from the above list we essentially remove the condition of
\( K \) to contain a field and give a direct elementary proof, but our conclusion of solvability  is weaker than triviality.

In a sequel paper we are going to prove solvability of locally isotropic
\( \KFunc_1 \) functor over finite dimensional rings. The proof of such solvability relies on the main result from this paper and certain localization technique, quite different from methods of this paper. Over arbitrary commutative rings locally isotropic
\( \KFunc_1 \) is constructed only for local isotropic rank at least
\( 2 \) in
\cite{ele-loc} (and the elementary subgroup of
\( \GenLin(1, K) \) is not necessarily normal), so our ultimate goal needs only group schemes of isotropic rank at least
\( 2 \) in the case of semilocal rings. During the proof below we also show that
\( \KFunc_1^G(K) \) is solvable for isotropic rank
\( 1 \) in a minor number of cases.

For classical Tits indices our proof relies on stability of
\( \KFunc_1 \) and basic properties of odd unitary groups. The whole argument seems to be well known to the experts, but we are unable to find a reference in the literature. The cases of exceptional Tits indices are reduced to the classical ones, sometimes of relative rank
\( 1 \). Such reduction is non-trivial for
\( \TitsIndex{E}{}{7}{2}{31} \). For appearing classical groups of isotropic rank
\( 1 \) we have to provide additional arguments not following from the general theory of odd unitary groups.

In this paper we do not try to find explicit bounds on the derived length of
\( \KFunc_1^G(K) \). Such bounds can be easily extracted from the proof. In particular, the derived length for all Tits indices from the main theorem is bounded by a common absolute constant.

The paper is organized as follows. In \S 2 we recall definitions of the isotropic rank and the elementary subgroup of an isotropic reductive group. We also prove in theorem
\ref{perfect} perfectness of the derived subgroup
\( [ \Elem_G(K), \Elem_G(K) ] \) for simple group scheme
\( G \) of isotropic rank at least
\( 2 \) over connected semilocal ring
\( K \) (even for Chevalley groups of types
\( \mathsf{B}_2 \) and
\( \mathsf{G}_2 \)). In \S 3 we collect various definitions of classical groups, they are needed to prove solvability of
\( \KFunc_1 \) for classical Tits indices. Theorem
\ref{twi-iso-cla} contains explicit construction of all isotropic simple group schemes (up to isogeny) over connected semilocal rings with classical Tits indices in terms of odd unitary groups. The last section \S 4 contains the proof of our main result, theorem \ref{sol-sem-iso}, and necessary preliminary lemmas.

The author wants to thank Anastasia Stavrova for giving motivation to work on this problem.

\section{Isotropic reductive groups}

In this paper we fix a commutative unital ring
\( K \), usually assumed to be semilocal. The term
\textit{root system} means a crystallographic possibly non-reduced root system, i.e.\ it can have components of type
\( \RootSys{BC}{\ell} \) for
\( \ell \geq 1 \). We say that roots of the largest length
\( \lambda \) in a given component are
\textit{long}, roots with length strictly between
\( \lambda / 2 \) and
\( \lambda \) are
\textit{short}, and for a component of type
\( \RootSys{BC}{\ell} \) roots with length
\( \lambda / 2 \) are
\textit{ultrashort}. So in the case of
\( \RootSys{BC}{1} \) there are only two long roots and two ultrashort ones. A basis of a root system
\( \Phi \) is usually denoted by
\( \Delta \) (it consists of short and ultrashort roots for components of type
\( \RootSys{BC}{\ell} \)), and
\( \Phi^{+} \),
\( \Phi^{-} \) are the sets of positive and negative roots with respect to a basis
\( \Delta \).

Now let
\( G \) be a reductive group scheme over
\( K \) in the sense of
\cite{red-grp-sch}, i.e.\ a smooth affine group schemes such that its geometric fibers
\( G_{ \overline{ \kappa( \mathfrak{p} ) } } \) are connected reductive algebraic groups over the algebraic closures
\( \overline{ \kappa( \mathfrak{p} ) } \) of all residue fields. We call
\( G \)
\textit{simple} if it is a twisted form of a Chevalley--Demazure group scheme
\( \ChevGrp^\Lambda_K( \widetilde{\Phi}, {-} ) \) for an irreducible
\textit{absolute root system}
\( \widetilde{\Phi} \) and some choice of a weight lattice
\( \Lambda \). Over fields such group schemes are often called absolutely almost simple.

Suppose that
\( K \) is semilocal with connected spectrum. Then
\( G \) has a maximal split torus and all such tori are conjugate by the group
\( G(K) \)
\cite[Exp. XXVI, proposition 6.16]{red-grp-sch}. Choose such a torus
\( T \leq G \). The weight decomposition of the Lie algebra
\( \liealg \) of
\( G \) with respect to the action of
\( T \) has the form
\[
	\liealg
	=
	\liealg_0
	\oplus
	\bigoplus_{ \alpha \in \Phi }
		\liealg_\alpha,
\]
where
\( \Phi \) is a (possibly non-reduced) root system called the
\textit{relative root system} in the character lattice of
\( T \) and all weight subspaces are non-zero free
\( K \)-modules. The rank of
\( \Phi \) is called the
\textit{isotropic rank} of
\( G \). There exists a faithfully flat extension
\( K \subseteq \widetilde{K} \) such that there is a maximal torus
\(
	T_{ \widetilde{K} }
	\leq
	\widetilde{T}
	\leq
	G_{ \widetilde{K} }
\) and
\( G_{ \widetilde{K} } \) has a splitting with the maximal torus
\( \widetilde{T} \). Without loss of generality
\( \widetilde{K} \) is also semilocal (possibly with disconnected spectrum), so there is the
\textit{absolute root system}
\( \widetilde{\Phi} \) in the character lattice of
\( \widetilde{T} \). The inclusion
\( T_{ \widetilde{K} } \to \widetilde{T} \) induces a map
\[
	u
	\colon
	\widetilde{\Phi} \sqcup \{ 0 \}
	\to
	\Phi \sqcup \{ 0 \}.
\]

This map
\( u \) is independent of the choices of
\( \widetilde{K} \) and
\( \widetilde{T} \) up to a linear isomorphism between the character lattices of the split tori. It is easy to see that there are canonical group subschemes
\( L \leq G \) and
\( U_\alpha \leq G \) for
\( \alpha \in \Phi \) such that
\begin{align*}
		L_{ \widetilde{K} }
		&=
		\bigl\langle
			\widetilde{T},
			\widetilde{U}_\beta
			\mid
			u(\beta) = 0
		\bigr\rangle,
	&
		( U_\alpha )_{ \widetilde{K} }
		&=
		\bigl\langle
			\widetilde{U}_\beta
			\mid
			u(\beta) \in \{ \alpha, 2 \alpha \}
		\bigr\rangle
\end{align*}
as fppf sheaves, where
\( \widetilde{U}_\beta \leq G_{ \widetilde{K} } \) are root subgroups associated with
\( \widetilde{T} \). In particular,
\( \liealg_0 \) is the Lie algebra of
\( L \),
\( \liealg_\alpha \) is the Lie algebra of
\( U_\alpha \) for non-ultrashort
\( \alpha \), and
\( \liealg_\alpha \oplus \liealg_{ 2 \alpha } \) is the Lie algebra of
\( U_\alpha \) otherwise. Actually,
\( L \) is the scheme centralizer of
\( T \), as a group scheme it is reductive.

Applying
\cite[Exp. XXVI, corollary 5.2]{red-grp-sch} (possibly to a subsystem reductive group subscheme of
\( G \)) we see that for every root
\( \alpha \) there exists an
\textit{%
	\( \alpha \)-Weyl element%
}, i.e.\ an element
\(
	w
	\in
	U_\alpha(K)\, U_{- \alpha}(K)\, U_\alpha(K)
\) such that
\begin{align*}
		\up{w}{L} &= L,
	&
		\up{w}{U_\beta} = U_{ s_\alpha(\beta) },
\end{align*}
where
\(
	s_\alpha(\beta)
	=
	\beta
	-
	2
	\frac{ ( \alpha, \beta ) }{ ( \alpha, \alpha ) }
	\alpha
\) is the reflection in the orthogonal complement of
\( \alpha \).

There is a complete description of all possible maps
\( u \colon \widetilde{\Phi} \to \Phi \) using Tits indices
\cite{tits-ind}. Namely, a Tits index
\( \TitsIndex{X}{s}{n}{r}{d} \) consists of the type
\( \RootSys{X}{n} \) of the absolute root system
\( \widetilde{\Phi} \), the rank
\( r \) of the relative root system
\( \Phi \), and two additional parameters
\( s \) and
\( d \). If
\( K \subseteq \widetilde{K} \) is a splitting extension for
\( G \) as above, then
\( G \) is determined by the descent data on the Chevalley group scheme
\(
	\ChevGrp_{ \widetilde{K} }^\Lambda(
		\widetilde{\Phi},
		{-}
	)
	\cong
	G_{ \widetilde{K} }
\), i.e.\ by an automorphism
\( \psi \) of
\(
	\ChevGrp_{
		\widetilde{K} \otimes_K \widetilde{K}
	}^\Lambda(
		\widetilde{\Phi},
		{-}
	)
\) satisfying the cocycle condition. This automorphism preserves all relative root subgroups
\(
	( U_\alpha )_{
		\widetilde{K} \otimes_K \widetilde{K}
	}
\) and the group subscheme
\( L_{ \widetilde{K} \otimes_K \widetilde{K} } \), so it is a
\( ( \widetilde{K} \otimes_K \widetilde{K} ) \)-point of the group subscheme
\[
		L_{ \widetilde{K} }
		\rtimes
		\Out( \widetilde{\Phi} )
	\leq
		\Aut\bigl(
			\ChevGrp^\Lambda_{ \widetilde{K} }(
				\widetilde{\Phi},
				{-}
			)
		\bigr)
	\cong
		\ChevGrp^{ \mathrm{ad} }_{ \widetilde{K} }(
			\widetilde{\Phi},
			{-}
		)
		\rtimes
		\Out( \widetilde{\Phi} ).
\]
The parameter
\( s \) denotes the size of the smallest subgroup
\( H \leq \Out(\widetilde{\Phi}) \) such that
\( \psi \) lies in
\( L_{ \widetilde{K} } \rtimes H \). The last parameter
\( d \) denotes the rank of an associated Azumaya algebra for classical Tits indices (and it is written inside brackets) and the dimension of the ``anisotropic kernel'' for exceptional Tits indices. Sometimes different Tits indices correspond to isomorphic maps
\( u \), a complete list of such coincidences can be found in
\cite[\S 2]{dio-pro}.

Now let
\(
	U^{\pm}
	=
	\langle U_\alpha \mid \alpha \in \Phi^{\pm} \rangle
\) be two unipotent group subschemes of
\( G \). They are the unipotent radicals of opposite minimal parabolic subgroups
\( P^{\pm} = U^{\pm} \rtimes L \) and
\( L \) is the common Levi subgroup of
\( P^{\pm} \) by
\cite[Exp. XXVI, proposition 6.16]{red-grp-sch}. Moreover, the group
\( G(K) \) admits the
\textit{Gauss decomposition}
\[
	G(K)
	=
	U^{+}(K)\, U^{-}(K)\, U^{+}(K)\, L(K)
\]
by
\cite[Exp. XXVI, corollary 5.2]{red-grp-sch}. It follows that the
\textit{elementary subgroup}
\(
	\Elem_G(K)
	=
	\langle U_\alpha(K) \mid \alpha \in \Phi \rangle
\) is normal in
\( G(K) \) and independent of the choice of
\( T \). The
\( \KFunc_1 \)-functor is
\[ \KFunc_1^G(K) = G(K) / \Elem_G(K). \]

\begin{theorem}
	\label{perfect}
	Let
	\( G \) be a simple group scheme over a commutative semilocal ring
	\( K \) of isotropic rank at least
	\( 2 \). Then the derived subgroup
	\( [ \Elem_G(K), \Elem_G(K) ] \) is perfect. If the Tits index is neither
	\(
		\TitsIndex{B}{}{2}{2}{}
		\cong
		\TitsIndex{C}{}{2}{2}{ (1) }
	\) nor
	\( \TitsIndex{G}{}{2}{2}{0} \), then
	\( \Elem_G(K) \) itself is perfect.
\end{theorem}
\begin{proof}
	The second claim is
	\cite[theorem 4]{ele-loc}. To prove the first case we only have to consider the Chevalley group schemes
	\( \ChevGrp^\Lambda( \RootSys{B}{2}, {-} ) \) and
	\( \ChevGrp^\Lambda( \RootSys{G}{2}, {-} ) \).
	
	\begin{figure}[ht]
		\centering
		\begin{tikzpicture}
				\foreach \i in { 0, 1, ..., 3 } {
					\draw[->]
						( 0, 0 )
						--
						( \i * 90 : { 2 / sqrt(2) } );
					\draw[->]
						( 0, 0 )
						--
						( 45 + \i * 90 : 2 );
				}
				\node
					[ right ]
					at ( 0 : { 2 / sqrt(2) } )
					{
						\( \alpha \)
					};
				\node
					[ left ]
					at ( 180 : { 2 / sqrt(2) } )
					{
						\( - \alpha \)
					};
				\node
					[ below ]
					at ( 270 : { 2 / sqrt(2) } )
					{
						\( \alpha - \beta \)
					};
				\node
					[ above right, inner sep = .2em ]
					at ( 45 : 2 )
					{
						\( \beta \)
					};
				\node
					[ above left, inner sep = .2em ]
					at ( 135 : 2 )
					{
						\( \beta - 2 \alpha \)
					};
				\node
					[ below left, inner sep = .2em ]
					at ( 225 : 2 )
					{
						\( - \beta \)
					};
				\tikzset{ shift = { ( 5, 0 ) } }
				\foreach \i in { 0, 1, ..., 5 } {
					\draw[->]
						( 0, 0 )
						--
						( \i * 60 : { 2 / sqrt(3) } );
					\draw[->]
						( 0, 0 )
						--
						( 30 + \i * 60 : 2 );
				}
				\node
					[ right ]
					at ( 0 : { 2 / sqrt(3) } )
					{
						\( \alpha \)
					};
				\node
					[ above right, inner sep = .2em ]
					at ( 60 : { 2 / sqrt(3) } )
					{
						\( \beta \)
					};
				\node
					[ above left, inner sep = .2em ]
					at ( 120 : { 2 / sqrt(3) } )
					{
						\( \beta - \alpha \!\!\! \!\!\! \)
					};
				\node
					[ left ]
					at ( 180 : { 2 / sqrt(3) } )
					{
						\( - \alpha \)
					};
		\end{tikzpicture}
		\caption{%
			Root systems
			\( \RootSys{B}{2} \) and
			\( \RootSys{G}{2} \)%
		}
	\end{figure}
	
	In the first case let
	\( \Elem'_G(K) \leq G(K) \) be the subgroup generated by
	\( t_\alpha(x y)\, t_\beta(x y^2) \) and
	\( t_\beta(2 y) \) for
	\( x, y \in K \),
	\( \angle(\alpha, \beta) = 45^\circ \),
	\( \alpha \) is short, and
	\( \beta \) is long. Clearly,
	\[
		\Elem'_G(K)
		=
		[ \Elem'_G(K), \Elem'_G(K) ]
		\leq
		[ \Elem_G(K), \Elem_G(K) ].
	\]
	On the other hand, we have
	\begin{align*}
			[ t_{ - \alpha }(x), t_\alpha(y) ]
			&=
			\up{ t_{ - \alpha }(x) }{
				\bigl[
						t_\alpha(y)\,
						t_\beta(y)
					,
						t_{ \beta - 2 \alpha }(x)
				\bigr]
			}\,
			\bigl[
					t_{ - \alpha }(x)\,
					t_{ \beta - 2 \alpha }(x)
				,
					t_\alpha(y)\,
					t_\beta(y)
			\bigr]\,
			\up{ t_\alpha(y) }{
				\bigl[
					t_\beta(y),
					t_{ - \alpha }(x)
				\bigr]
			},
		\\
			[ t_{ - \beta }(x), t_\beta(y) ]
			&=
			\up{ t_{ - \beta }(x) }{
				\bigl[
					t_\beta(y),
					t_{ \alpha - \beta }(x)
				\bigr]
			}\,
			\bigl[
					t_{ - \beta }(x)\,
					t_{ \alpha - \beta }(x)
				,
					t_\beta(y)\,
					t_\alpha(y)
			\bigr]\,
			\up{ t_\beta(y) }{
				\bigl[
						t_\alpha(y)
					,
						t_{ - \beta }(x)\,
						t_{ \alpha - \beta }(x)
				\bigr]
			}
	\end{align*}
	with
	\( \alpha \) and
	\( \beta \) as above. It easily follows that
	\( \Elem'_G(K) \) is normalized by
	\( \Elem_G(K) \) and all generators
	\( [ t_\gamma(x), t_\delta(y) ] \) of
	\( [ \Elem_G(K), \Elem_G(K) ] \) as a normal subgroup of
	\( \Elem_G(K) \) lie in
	\( \Elem'_G(K) \). Therefore
	\( \Elem'_G(K) = [ \Elem_G(K), \Elem_G(K) ] \).
	
	In the second case we similarly define
	\( \Elem'_G(K) \) as the subgroup generated by
	\( t_\alpha(2 x) \),
	\( t_\alpha(x y)\, t_\beta(x y^2) \) for short roots
	\( \alpha \) and
	\( \beta \) with
	\( \angle(\alpha, \beta) = 60^\circ \) and by
	\( t_\gamma(z) \) for long
	\( \gamma \). The corresponding calculation is
	\[
		[ t_{ - \alpha }(x), t_\alpha(y) ]
		=
		\up{ t_{ - \alpha }(x) }{
			\bigl[
					t_\alpha(y)\,
					t_\beta(y)
				,
					t_{ \beta - \alpha }(x)
			\bigr]
		}\,
		\bigl[
				t_{ - \alpha }(x)\,
				t_{ \beta - \alpha }(x)
			,
				t_\alpha(y)\,
				t_\beta(y)
		\bigr]\,
		\up{ t_\alpha(y) }{
			\bigl[
				t_\beta(y),
				t_{ - \alpha }(x)
			\bigr]
		}
	\]
	for short roots
	\( \alpha \),
	\( \beta \) at the angle
	\( 60^\circ \).
\end{proof}

If
\( C \leq G \) is a central group subscheme of multiplicative type, then
\( P \mapsto P / C \) is a one-to-one correspondence between parabolic subgroups of
\( G \) and
\( G / C \). It follows that the isotropic ranks of
\( G \) and
\( G / C \) coincide and
\( \Elem_{ G / C }(K) \leq ( G / C )(K) \) is the image of
\( \Elem_G(K) \leq G(K) \).

The following lemma allows us to consider simple group schemes up to isogeny in proving solvability of
\( \KFunc_1^G(K) \). Recall that a
\textit{braided crossed module} consists of a group homomorphism
\( d \colon X \to H \), an action of
\( H \) on
\( X \) by group automorphisms, and a map
\( \langle {-}, {=} \rangle \colon H \times H \to X \) (a
\textit{crossed pairing}) such that
\begin{align*}
		d( \up{h}{x} ) &= \up{h}{ d(x) },
	&
		\langle h h', h'' \rangle
		&=
		\up{h}{ \langle h', h'' \rangle }\,
		\langle h, h'' \rangle,
	\\
		\up{x}{y} &= \up{ d(x) }{y},
	&
		\langle h, h' h'' \rangle
		&=
		\langle h, h' \rangle\,
		\up{h'}{ \langle h, h'' \rangle },
	\\
		d\bigl( \langle h, h' \rangle \bigr)
		&= 
		[h, h'],
	&
		\langle h, d(x) \rangle
		&=
		\up{h}{x}\, x^{- 1},
	\\
		\up{h}{ \langle h', h'' \rangle }
		&=
		\bigl\langle
			\up{h}{h'},
			\up{h}{h''}
		\bigr\rangle,
	&
		\langle d(x), h \rangle
		&=
		x\, ( \up{h}{x} )^{- 1}
\end{align*}
for
\( x, y \in X \) and
\( h, h', h'' \in H \).

\begin{lemma}
	\label{sol-iso}
	Let
	\( G \) be a reductive group scheme over arbitrary
	\( K \) of local isotropic rank at least
	\( 2 \). Let also
	\( C \leq G \) be a central group subscheme of multiplicative type. Then
	\[
		\KFunc_1^G(K) \to \KFunc_1^{ G / C }(K)
	\]
	is a braided crossed module in a canonical way. In particular,
	the derived subgroup
	\(
		\bigl[
			\KFunc_1^{ G / C }(K),
			\KFunc_1^{ G / C }(K)
		\bigr]
	\) contains the image of
	\( \KFunc_1^G(K) \) and the kernel of the homomorphism
	\( \KFunc_1^G(K) \to \KFunc_1^{ G / C }(K) \) is central in
	\( \KFunc_1^G(K) \).
\end{lemma}
\begin{proof}
	Clearly, if
	\( C \leq G \) is a central subgroup of any group, then
	\( G \to G / C \) is a braided crossed module in a unique way. This argument actually works for group objects in any Barr exact category, in particular, in the category of fppf sheaves on the category of affine
	\( K \)-schemes.
	
	Now let
	\( G \) be the group scheme from the statement. We know that
	\( G \to G / C \) is a braided crossed module of affine group schemes in a unique way. Then
	\( G(K) \to (G / C)(K) \) is canonically a braided crossed module of ordinary groups.
	
	It remains to check that all operations are well defined on the quotients
	\( \KFunc_1^G(K) \) and
	\( \KFunc_1^{ G / C }(K) \). This easy follows from surjectivity of
	\( \Elem_G(K) \to \Elem_{ G / C }(K) \) and invariance of
	\( \Elem_G(K) \leq G(K) \) under the group
	\( \Aut(G) \) of algebraic automorphisms of
	\( G \).
\end{proof}

\section{Constructions of classical groups}

Before proceeding with solvability of
\( \KFunc_1 \) over semilocal rings we need explicit construction of all classical simple group schemes up to isogeny.

Let
\( S \) be a unital associative ring and
\( \lambda \in S^* \) an invertible element. A
\textit{%
	\( \lambda \)-involution%
} on
\( S \) is an additive map
\( S \to S,\, x \mapsto x^* \) such that
\( 1^* = 1 \),
\( (x y)^* = y^* x^* \) (i.e.\ the map is an anti-endomorphism),
\( x^{* *} = \lambda x \lambda^{- 1} \), and
\( \lambda^* = \lambda^{- 1} \). In this paper we need only the cases
\( \lambda = \pm 1 \), so the reader can assume that
\( \lambda \) is central. A
\textit{form parameter} is an additive subgroup
\( \Lambda \leq S \) such that
\( x^* \Lambda x \leq \Lambda \) for all
\( x \in S \) and
\[
	\{ x - x^* \lambda \mid x \in S \}
	=
	\Lambda_{\min}
	\leq
	\Lambda
	\leq
	\Lambda_{\max}
	=
	\{ x \in S \mid x + x^* \lambda = 0 \}.
\]
Clearly, both
\( \Lambda_{\min} \) and
\( \Lambda_{\max} \) are form parameters and they coincide if
\( 2 \in S^* \) (in this case there is only one form parameter). A
\textit{form ring} is a unital associative ring
\( S \) together with an element
\( \lambda \in S^* \), a
\( \lambda \)-involution, and a form parameter.

A
\textit{hermitian form} on a right module
\( M \) over a form ring
\( S \) is a biadditive map
\( B \colon M \times M \to S \) such that
\( B(m x, n y) = x^* B(m, n) y \) and
\( B(m, n) = B(n, m)^* \lambda \). Such a form is called
\textit{non-degenerate} if
\( M_S \) is finitely generated projective
\( m \mapsto B( m, {-} ) \) is a bijection between
\( M \) and the dual module
\( \Hom_S(M, S) \). A
\textit{quadratic form} is a map
\( q \colon M \to S / \Lambda \) such that
\begin{align*}
		q(m x) &= x^* q'(m) x + \Lambda,
	&
		q(m + n) &= q(m) + B(m, n) + q(m'),
	&
		B(m, m) &= q'(m) + q'(m)^* \lambda,
\end{align*}
where
\( q'(m) \in S \) denotes any fixed preimage of
\( q(m) \). The triple
\( (M, B, q) \) is called a
\textit{quadratic module}. The
\textit{unitary group} of a quadratic module
\( (M, B, q) \) is
\[
	\Unit(M)
	=
	\{
		g \in \Aut(M_S)
		\mid
		B(g m, g n) = B(m, n),
		q(g m) = q(m)
	\}.
\]

There is a general method to construct quadratic forms. Let
\( Q \colon M \times M \to S \) be a
\textit{sesquilinear form} on a right
\( S \)-module
\( M \), i.e.\ a biadditive map such that
\( Q(m x, n y) = x^* Q(m, n) y \). Then
\( (M, B, q) \) is a quadratic module, where
\begin{align*}
		B(m, n) &= Q(m, n) + Q(n, m)^* \lambda,
	&
		q(m) &= Q(m, m) + \Lambda.
\end{align*}

\begin{lemma}
	\label{qua-eve}
	Suppose that
	\( M \) is a finitely generated projective
	\( S \)-module. Then every pair of hermitian and quadratic forms on
	\( M \) can be constructed by a sesquilinear form.
\end{lemma}
\begin{proof}
	Let
	\( P \) be a finitely generated projective module such that
	\( M \oplus P \cong S^n \) is free. The module
	\( M \oplus P \) has forms
	\begin{align*}
			B( m \oplus p, m' \oplus p' ) &= B(m, n),
		&
			q( m \oplus p ) &= q(m),
	\end{align*}
	i.e.\ this is the orthogonal sum of
	\( M \) and
	\( P \), and the forms on
	\( P \) are zero. Clearly, it suffices to construct appropriate sesquilinear form on
	\( M \oplus P \), so without loss of generality
	\( M = S^n \) is free.
	
	Now let
	\( e_1, \ldots, e_n \in M \) be its basis as a free module. Both forms are completely determined by the values
	\( b_{i j} = B(e_i, e_j) \in S \) and
	\( q_i = q(e_i) \in S / \Lambda \). Moreover, these values can be arbitrary subject to the relations
	\begin{align*}
			b_{i j} &= b_{j i}^* \lambda
			\text{ for }
			i > j,
		&
			b_{i i} &= q'_i + (q'_i)^* \lambda
	\end{align*}
	where
	\( q'_i \in S \) are arbitrary preimages of
	\( q_i \). We can take
	\[
		Q\Bigl(
			\bigoplus_{i = 1}^n e_i x_i,
			\bigoplus_{i = 1}^n e_i y_i
		\Bigr)
		=
		\sum_{ 1 \leq i < j \leq n }
			x_i^* b_{i j} y_j
		+
		\sum_{i = 1}^n
			x_i^* q'_i y_i.
		\qedhere
	\]
\end{proof}

For example, let
\( P \) be a finitely generated projective module over a form ring
\( S \). The abelian group
\( \Hyp(P) = \Hom_S(P, S) \oplus P \) is a right
\( S \)-module under
\( (f \oplus p) x = x^* f \oplus p x \). Together with the forms
\begin{align*}
		B( f \oplus p, f' \oplus p' )
		&=
		f(p') + f'(p)^* \lambda,
	&
		q(f \oplus p) &= f(p) + \Lambda
\end{align*}
the module
\( \Hyp(P) \) is called a
\textit{hyperbolic quadratic module}. These forms are constructed by the sesquilinear form
\( Q( f \oplus p, f' \oplus p' ) = f(p') \), and the hermitian form is non-degenerate.

As an example of possible degenerate hermitian forms let
\( S = K \) be commutative with
\( \lambda = 1 \),
\( x^* = x \), and
\( \Lambda = 0 \). We say that
\( q \colon M \to K \) is a
\textit{traditional quadratic form} if
\( q(m x) = q(m) x^2 \) and the expression
\( B(m, n) = q(m + n) - q(m) - q(n) \) is bilinear. Clearly, all quadratic modules over this
\( S \) are precisely modules with traditional quadratic forms. The corresponding unitary group are called
\textit{orthogonal group} and denoted by
\( \Orth(q) \). If
\( M = \bigoplus_{i = 1}^{n} e_i K \) is free of finite rank, then
\( B \) is non-degenerate if and only if its
\textit{discriminant}
\( \mathrm{disc}(B) = \det_{i, j = 1}^n B(e_i, e_j) \) is invertible. In the case of odd rank
\( n = 2 \ell + 1 \) this determinant as a polynomial over
\( \Int \) in the formal variables
\( B(e_i, e_j) = B(e_j, e_i) \) for
\( i < j \) and
\( q(e_i) = \frac{1}{2} B(e_i, e_i) \) is divisible by
\( 2 \) and we call
\( \hdisc(B) = \frac{1}{2} \mathrm{disc}(B) \) the
\textit{half-discriminant} of
\( B \). A quadratic form
\( q \) on a free module of odd rank is called
\textit{semi-regular} if the corresponding half-discriminant is invertible. Both discriminant and half-discriminant preserve invertibility under base change, so we can define semi-regular quadratic forms on finitely generated projective modules of constant odd rank by Zariski localization. If
\( 2 \) is not invertible, then semi-regular quadratic forms always have degenerate symmetric bilinear forms
\( B \).

Sometimes it is more convenient to work with the graph of a quadratic form
\( q \) instead of
\( q \) itself. The
\textit{Heisenberg group} of a hermitian form
\( B \) is the set
\( \Heis(B) = M \times S \) with the group operation
\[
	(m, x) \dotplus (n, y)
	=
	( m + n, x - B(m, n) + y ).
\]
The multiplicative monoid
\( S^\bullet\) of
\( S \) acts on
\( \Heis(B) \) by endomorphisms from the right,
\( (m, x) \cdot y = ( m y, y^* x y ) \). An
\textit{odd form parameter}
\cite{odd-uni-gro} is an
\( S^\bullet \)-invariant subgroup
\( \OddFormPar \leq \Heis(B) \) such that
\[
	\{ ( 0, x - x^* \lambda ) \mid x \in S \}
	=
	\OddFormPar_{\min}
	\leq
	\OddFormPar
	\leq
	\OddFormPar_{\max}
	=
	\{ (m, x) \mid x + B(m, m) + x^* \lambda = 0 \}.
\]
Clearly, both
\( \OddFormPar_{\min} \) and
\( \OddFormPar_{\max} \) are odd form parameters, i.e.\ they are
\( S^\bullet \)-invariant subgroups and
\( \OddFormPar_{\min} \leq \OddFormPar_{\max} \). In this case the unitary group is defined as
\[
	\Unit(M)
	=
	\bigl\{
		g \in \Aut(M_S)
		\mid
		B(g m, g n) = B(m, n),
		( g m - m, B(g m - m, m) ) \in \OddFormPar
	\bigr\}.
\]

If
\( \Lambda \) is an ordinary form parameter and
\( q \colon M \to S / \Lambda \) is a quadratic form, then
\( \OddFormPar = \{ (m, x) \mid x + \Lambda = - q(m) \} \) is an odd form parameter. Conversely, every odd form parameter
\( \OddFormPar \) such that the first projection
\( \OddFormPar \to M \) is surjective appears in this way from
\(
	\Lambda
	=
	\{ x \in S \mid (0, x) \in \OddFormPar \}
\) and a unique quadratic form.

The language of ordinary form parameters and quadratic forms allows us to easily construct orthogonal sums of modules with forms. On the other hand, the language of odd form parameters is better suited to work with elementary unitary groups.

Now let
\( (M, B, q) \) be a quadratic module over a form ring
\( S \). Take an integer
\( \ell \geq 0 \). Consider the module
\[
	e_{- \ell} S
	\oplus
	\ldots
	\oplus
	e_{- 1} S
	\oplus
	M
	\oplus
	e_1 S
	\oplus
	\ldots
	\oplus
	e_\ell S
\]
obtained from
\( M \) by adding a free direct summand of rank
\( 2 \ell \). This large module has the forms
\begin{align*}
		B\Bigl(
				\bigoplus_{i = - \ell}^{- 1}
					e_i x_i
				\oplus
				m
				\oplus
				\bigoplus_{i = 1}^{\ell}
					e_i x_i
			,
				\bigoplus_{i = - \ell}^{- 1}
					e_i y_i
				\oplus
				n
				\oplus
				\bigoplus_{i = 1}^{\ell}
					e_i y_i
		\Bigr)
		&=
		\sum_{i = - \ell}^{- 1}
			x_i^* y_{- i}
		+
		B(m, n)
		+
		\sum_{i = 1}^\ell
			x_i^* \lambda y_{- i},
	\\
		q\Bigl(
			\bigoplus_{i = - \ell}^{- 1}
				e_i x_i
			\oplus
			m
			\oplus
			\bigoplus_{i = 1}^{\ell}
				e_i x_i
		\Bigr)
		&=
		q(m)
		+
		\sum_{i = - \ell}^{- 1}
			x_i^* x_{- i}.
\end{align*}
In other words, this is the orthogonal sum of
\( M \) and the split hyperbolic quadratic module
\(
	\Hyp\bigl(
		\bigoplus_{i = 1}^\ell
			e_i S
	\bigr)
\). We denote the resulting unitary group by
\( \Unit(2 \ell, M) \). Clearly, there is a natural chain
\[
	\Unit(M)
	=
	\Unit(0, M)
	\leq
	\Unit(2, M)
	\leq
	\Unit(4, M)
	\leq
	\ldots
\]
of groups. Starting from
\( \ell = 1 \) each of these groups has the canonical elementary subgroup
\( \ElemUnit(2 \ell, M) \leq \Unit(2 \ell, M) \) generated by strictly upper and lower triangular generalized matrices from
\( \Unit(2 \ell, M) \). The factor-set is denoted by
\( \KUnit_1(2 \ell, M) \). By a special case of
\cite[theorem 4]{odd-uni-gro} the elementary subgroup is normal if
\( S \) is semilocal and
\( \ell \geq 2 \), so
\( \KUnit_1(2 \ell, M) \) is a group under these conditions. As with isotropic reductive groups, elementary subgroups of unitary groups are also generated by their
\textit{root subgroups} indexed by a root system of type
\( \RootSys{BC}{\ell} \), where ultrashort root subgroups are isomorphic to the odd form parameter
\( \OddFormPar \), short ones are isomorphic to the ring
\( S \), and the long ones are isomorphic to the ordinary form parameter
\( \Lambda \). See
\cite[\S 6]{odd-uni-gro} for details.

Finally, there is a construction of unitary groups avoiding quadratic modules. A pair
\( (R, \Delta) \) is called a (special unital) odd form ring if
\( R \) is a unital associative ring with involution
\( x \mapsto \inv{x} \) (with
\( \lambda = 1 \)) and
\( \Delta \) is an odd form parameter of the hermitian form
\( (x, y) \mapsto \inv{x} y \) on the module
\( R \). In other words,
\( \Delta \subseteq R \times R \) is closed under the operations
\begin{align*}
		(x, y) \dotplus (z, w)
		&=
		( x + z, y - \inv{x} z + w ),
	&
		(x, y) \cdot z &= ( x z, \inv{z} y z ),
\end{align*}
and
\[
	\{ ( 0, x - \inv{x} ) \mid x \in R \}
	=
	\Delta_{\min}
	\leq
	\Delta
	\leq
	\Delta_{\max}
	=
	\{
		(x, y) \in R \times R
		\mid
		y + \inv{x} x + \inv{y} = 0
	\}.
\]
The unitary group of
\( R \) considered as a quadratic module is
\[
	\Unit(R, \Delta)
	=
	\{
		g \in R^*
		\mid
		g^{- 1} = \inv{g},
		(g - 1, g - 1) \in \Delta
	\}.
\]
In
\cite{twi-for-cla} we used the condition
\( ( g - 1, \inv{g} - 1 ) \in \Delta \), but this is equivalent to
\( ( g - 1, g - 1 ) \in \Delta \) because
\( ( 0, g - \inv{g} ) \in \Delta_{\min} \).

An
\textit{orthogonal hyperbolic family} of rank
\( \ell \) in an odd form ring
\( (R, \Delta) \) is a complete family of orthogonal idempotents
\( e_{- \ell}, \ldots, e_0, \ldots, e_\ell \in R \) such that
\( \inv{e_i} = e_{- i} \) and
\( (e_i, 0) \in \Delta \) for
\( i \neq 0 \). Such a family is a replacement of a family of hyperbolic direct summand from the definition of
\( \Unit(2 \ell, M) \) above, in particular, the elementary subgroup
\( \ElemUnit(R, \Delta) \leq \Unit(R, \Delta) \) is defined as the group generated by upper and lower generalized triangular matrices from
\( \Unit(R, \Delta) \). If such a family is fixed, then root elements are defined as
\begin{align*}
		T_{i j}(x) &= 1 + x - \inv{x},
	&
		T_i(y, z) &= 1 + y + z - \inv{y},
\end{align*}
where
\( x \in e_i R e_j \),
\( y \in e_0 R e_i \),
\( z \in e_{- i} R e_i \),
\( (y, z) \in \Delta \), and
\( 0 \neq i \neq \pm j \neq 0 \). The root subgroup
\( \Image( T_{i j} ) = \Image( T_{- j, - i} ) \) corresponds to the root
\( \e_j - \e_i \), and
\( \Image( T_i ) \) corresponds to the root
\( \e_i \) of
\[
		\Phi
	=
		\{
			\pm \e_i \pm \e_j
			\mid
			1 \leq i < j \leq \ell
		\}
		\sqcup
		\{
			\pm \e_i,
			\pm 2 \e_i
			\mid
			1 \leq i \leq \ell
		\}
	\subseteq
		\Real^\ell,
\]
where
\( \e_{- i} = - \e_i \). The long root subgroup corresponding to
\( 2 \e_i \) consists of all
\( T_i(0, z) \).

Unitary groups constructed by quadratic modules and odd form algebras are related as follows. Let
\( S \) be a form ring and
\( (M, B, q) \) a quadratic module over
\( S \) with non-degenerate hermitian form
\( B \). Then
\( R = \mathrm{End}(M_S) \) has the involution
\( B( \inv{r} m, n ) = B(m, r n) \) and the form parameter
\[
	\Delta
	=
	\{
		(r, s) \in R \times R
		\mid
		s + \inv{r} r + \inv{s} = 0,
		q(r m) + B(m, s m) = 0
	\},
\]
so
\( (R, \Delta) \) is an odd form ring
\cite[\S 3]{twi-for-cla}. It is easy to see that
\( \Unit(M) = \Unit(R, \Delta) \) as subgroups of
\( R \). If we apply this construction to the quadratic module
\( M \perp \Hyp(S^\ell) \) instead of
\( S \), then the resulting endomorphism ring
\( R \) has an obvious orthogonal hyperbolic family.

Before stating main results we need to deal with triality in the case
\( \RootSys{D}{4} \). Recall that the split simple adjoint group scheme
\( \ProjSpecOrthSch_8 \) of type
\( \RootSys{D}{4} \) over
\( K \) has the automorphism group scheme
\( \ProjSpecOrthSch_8 \rtimes \Sym_3 \), where the symmetric group
\( \Sym_3 \) is considered as a constant sheaf. We say that a simple group scheme
\( G \) with absolute root system of type
\( \RootSys{D}{4} \)
\textit{avoids triality} if its isomorphism class lies in the image of
\[
	\Hyp^1_\fppf(
		K,
		\ProjSpecOrthSch_8 \rtimes \Int / 2 \Int
	)
	\to
	\Hyp^1_\fppf(
		K,
		\ProjSpecOrthSch_8 \rtimes \Sym_3
	).
\]
This is independent of the choice of a subgroup
\( \Int / 2 \Int \leq \Sym_3 \) because all such subgroups are conjugate. For example,
\( G \) avoids triality if it is an inner form (from
\( \Hyp^1_\fppf( K, \ProjSpecOrthSch_8 ) \)) or a twisted form of
\( \SpecOrthSch_8 \). Also
\( G \) avoids triality if the base ring
\( K \) is semilocal with connected spectrum, and the Tits index is classical (i.e.\ neither
\( \TitsIndex{D}{s}{4}{0}{28} \), nor
\( \TitsIndex{D}{s}{4}{1}{9} \), nor
\( \TitsIndex{D}{s}{4}{2}{2} \) for
\( s \in \{ 3, 6 \} \)).

\begin{lemma}
	\label{twi-cla}
	Let
	\( G \) be a simple adjoint group scheme over arbitrary commutative unital ring
	\( K \) with the absolute root system of classical type, i.e.\ %
	\( \RootSys{A}{\ell} \),
	\( \RootSys{B}{\ell} \),
	\( \RootSys{C}{\ell} \), or
	\( \RootSys{D}{\ell} \). In the case of
	\( \RootSys{D}{4} \) assume that
	\( G \) avoids triality. If its absolute root system is of type
	\( \RootSys{B}{\ell} \), then
	\( G \) is the orthogonal group scheme of a finitely generated projective
	\( K \)-module of constant rank
	\( 2 \ell + 1 \) with semi-regular traditional quadratic form. If the type is
	\( \RootSys{A}{\ell} \),
	\( \RootSys{C}{\ell} \), or
	\( \RootSys{D}{\ell} \), then
	\( G \) is the scheme derived subgroup of the automorphism group scheme of a sheaf
	\( ( \ModSheaf(R), \ModSheaf(\Delta) ) \) of odd form
	\( K \)-algebras,
	\( R \) is an Azumaya algebra over
	\( K \) or its quadratic \'etale extension, the first projection
	\( \Delta \to R \) is surjective, and the kernel of this projection is finitely generated projective
	\( K \)-module. The scheme derived subgroup of the unitary group sheaf is also a simple group scheme and a finite central extension of
	\( G \).
\end{lemma}
\begin{proof}
	Recall that the group scheme
	\( \SpecOrthSch_{2 \ell + 1} \) for
	\( \ell \geq 1 \) is simple and adjoint, and its point group
	\( \SpecOrth(2 \ell + 1, K) \) is the intersection of the orthogonal group
	\( \Orth(2 \ell + 1, K) \) of the split traditional quadratic form
	\(
		q( x_{- \ell}, \ldots, x_\ell )
		=
		x_0^2 + \sum_{i = 1}^\ell x_{- i} x_i
	\) with
	\( \SpecLin(2 \ell + 1, K) \). In other words,
	\( \SpecOrth(2 \ell + 1, K) \) is the automorphism group of the triple
	\( (M, q, \omega) \), where
	\( M = K^{2 \ell + 1} \) and
	\(
		\omega
		=
		e_{- \ell} \wedge \ldots \wedge e_\ell
		\in
		\Lambda^{2 \ell + 1}(M)
	\) is the standard volume form. Since
	\( \SpecOrthSch_{2 \ell + 1} \) is its own automorphism group scheme, all its twisted forms are automorphism group schemes of triples
	\( (M, q, \omega) \), where
	\( M \) is a projective
	\( K \)-module of constant rank
	\( 2 \ell + 1 \),
	\( q \colon M \to K \) is a traditional semi-regular quadratic form on
	\( M \), and
	\( \omega \in \Lambda^{2 \ell + 1}(M) \) is a non-degenerate volume form such that
	\( \hdisc_\omega(q) = (- 1)^\ell \). Here
	\( \hdisc_\omega \) denotes the half-discriminant evaluated Zariski locally in a basis with the exterior product
	\( \omega \).
	
	Remaining cases follow from
	\cite[theorem 1]{twi-for-cla}. If
	\( G \) is a twisted form of
	\( \ProjGenLin_2 \), then it is the automorphism group scheme of a quaternion algebra
	\( R \) over
	\( K \) and the unitary group scheme of
	\( R \times R^{ \mathrm{op} } \) with the maximal odd form parameter.
\end{proof}

In the next theorem the three cases are distinguished by the type of the relative root system, namely,
\( \RootSys{A}{r} \) in the first case,
\( \RootSys{B}{r} \) or
\( \RootSys{D}{r} \) in the second one, and
\( \RootSys{BC}{r} \) or
\( \RootSys{C}{r} \) in the last one. The result (with the same proof) actually hold for simple adjoint group schemes over arbitrary commutative ring if we require that the descent data is contained in the group subscheme corresponding to a Tits index and there exist all Weyl element permuting the root group subschemes of
\( G \) and normalizing the group subscheme
\( L \leq G \).

\begin{theorem}
	\label{twi-iso-cla}
	Let
	\( G \) be a simple adjoint group scheme over semilocal
	\( K \) with connected spectrum. Assume that the isotropic rank of
	\( G \) is positive and the Tits index is classical (in particular, neither
	\( \TitsIndex{D}{s}{4}{1}{9} \) nor
	\( \TitsIndex{D}{s}{4}{2}{2} \)). Then
	\( G \) has one of the following standard forms.
	\begin{itemize}
		
		\item
		For Tits indices
		\( \TitsIndex{A}{1}{n}{r}{ (d) } \) the group scheme is the automorphism group scheme of the algebra
		\( R = \Mat(\ell + 1, S) \), where
		\( S \) is an Azumaya algebra of rank
		\( d \).
		
		\item
		For Tits indices
		\( \TitsIndex{B}{}{n}{r}{} \) and
		\( \TitsIndex{D}{s}{n}{r}{ (1) } \) the group scheme is the projective special orthogonal group
		\(
			\ProjSpecOrthSch(r, M)
			=
			\ProjSpecOrthSch( M \perp \Hyp(K^r) )
		\), where
		\( M \) is finitely generated projective module of constant rank
		\( 2 n + 1 - 2 r \) with semi-regular traditional quadratic form or
		\( 2 n - 2 r \) with non-degenerate traditional quadratic form respectively.
		
		\item
		Otherwise if
		\( \TitsIndex{X}{s}{n}{r}{ (d) } \) is the Tits index, then there are a form ring
		\( S \) and a quadratic module
		\( (M, B, q) \) over
		\( S \) such that
		\( S \) is an Azumaya algebra over
		\( K \) or its quadratic \'etale extension (for
		\( \mathsf{X} = \mathsf{A} \)) of rank
		\( d \) with
		\( K \)-linear involution,
		\( \lambda = - 1 \), the form parameter
		\( \Lambda \) is a direct summand of
		\( S \) as a
		\( K \)-module, and the hermitian form
		\( B \) is non-degenerate. The rank of
		\( M \) over
		\( K \) is
		\( 2 d (n - r d) \) for
		\(
			\mathsf{X}
			\in
			\{ \mathsf{C}, \mathsf{D} \}
		\) and
		\( 2 d (n + 1 - 2 r d) \) for
		\( \mathsf{X} = \mathsf{A} \). Let
		\( \UnitSch \) be the unitary group scheme of
		\( M \perp \Hyp(S^r) \). Then the scheme derived subgroup
		\( [ \UnitSch, \UnitSch ] \) is a simple reductive group scheme and
		\( G \) is its scheme factor-group by the center.
		
	\end{itemize}
	In all these cases the scheme images in
	\( G \) of the group schemes of upper and lower triangular matrices are opposite minimal parabolic subgroups.
\end{theorem}
\begin{proof}
	In the first case
	\( G \) is an inner twisted form of
	\( \ProjGenLin_{n + 1} \), so it is the automorphism group scheme of an Azumaya algebra
	\( R \). Existence of a split torus with root subgroups means that
	\( R \) is a generalized matrix ring, i.e.\ there is a complete family of orthogonal idempotents
	\( e_{0 0}, \ldots, e_{n n} \in R \). Moreover, existence of Weyl elements for basic roots easily implies that there are
	\begin{align*}
			e_{i, i + 1}
			&\in
			e_{i i} R e_{i + 1, i + 1},
		&
			e_{i + 1, i}
			&\in
			e_{i + 1, i + 1} R e_{i i}
	\end{align*}
	for
	\( 0 \leq i < n \) such that
	\begin{align*}
			e_{i, i + 1} e_{i + 1, i}
			&=
			e_{i i},
		&
			e_{i + 1, i} e_{i, i + 1}
			&=
			e_{i + 1, i + 1}.
	\end{align*}
	It follows that
	\( R \cong \Mat(n + 1, S) \) for
	\( S = e_{0 0} R e_{0 0} \).
	
	Now consider the case
	\( \TitsIndex{B}{}{n}{r}{} \). By lemma
	\ref{twi-cla} we have
	\( G = \SpecOrthSch(M) \) for some projective module
	\( M \) of constant rank
	\( 2 \ell + 1 \) with semi-regular traditional quadratic form
	\( q \). Existence of root subgroups induces a decomposition
	\( M = \bigoplus_{i = - r}^r M_i \), where
	\( M_i \) have rank
	\( 1 \) for
	\( i \neq 0 \),
	\( M_0 \) has rank
	\( 2 n - 2 r + 1 \),
	\( M_i \perp M_j \) unless
	\( i = - j \),
	\( q(M_i) = 0 \), and the symmetric bilinear form
	\( B \) induces isomorphisms
	\( M_{- i} \cong \Hom_K(M_i, K) = M_i^{\vee} \) for
	\( i \neq 0 \). As in the first case, existence of basic Weyl elements with long roots means that up to isomorphism
	\( M_1 = \ldots = M_r = L \) and
	\( M_{- 1} = \ldots = M_{- r} = L^{\vee} \) for a projective module
	\( L \) of constant rank
	\( 1 \). Finally, short basic root subgroups are parameterized by
	\( M_0 \otimes_K L \) and
	\( M_0 \otimes_K L^{\vee} \), and short Weyl elements exists if and only if there is
	\(
		x
		=
		\sum_{i = 1}^N
			m_i \otimes l_i
		\in
		M_0 \otimes_K L
	\) such that
	\[
			\widehat{q}(x)
		=
			\sum_{i = 1}^N
				q(m_i)\, l_i \otimes l_i
			+
			\sum_{ 1 \leq i < j \leq N }
				B(m_i, m_j)\, l_i \otimes l_j
		\in
			L \otimes_K L
	\]
	is invertible. Existence of such element means that
	\( L \otimes_K L \) has a basis, so
	\( L \cong L^{\vee} \). Finally, we can replace
	\( M \) with
	\( M \otimes_K L \) and
	\( q \) with
	\( \widehat{q} \) via the above formula to make all
	\( M_i \) free modules for
	\( i \neq 0 \). Such a replacement clearly preserves the odd orthogonal group.
	
	In the remaining cases let
	\( (R, \Delta) \) be the associated odd form ring from lemma
	\ref{twi-cla}. Below we use that it is a twisted form of one of the known ``split'' odd form algebras from
	\cite[\S 5]{twi-for-cla}, i.e.\ the ones constructed by the corresponding split quadratic modules. The set of root subgroups corresponds to an orthogonal hyperbolic family
	\( e_{- r}, \ldots, e_0, \ldots, e_r \in R \), where
	\( e_i R e_i \) is an Azumaya algebra over
	\( K \) of rank
	\( d \) for
	\( \mathsf{X} \in \{ \mathsf{C}, \mathsf{D} \} \) and an Azumaya algebra over a quadratic \'etale extension of
	\( K \) of rank
	\( d \) for
	\( \mathsf{X} = \mathsf{A} \). As in the first case, existence of Weyl elements with short roots (or long roots for
	\( \TitsIndex{D}{s}{n}{r}{ (1) } \)) means that there are
	\( e_{i j} \in e_i R e_j \) such that
	\( e_i = e_{i i} \) and
	\( e_{i j} e_{j k} = e_{i k} \) for
	\( i \),
	\( j \),
	\( k \) simultaneously positive or negative.
	
	Assume that the Tits index is
	\( \TitsIndex{D}{s}{n}{r}{ (1) } \). Then
	\( e_i R e_i \cong K \) and all root elements
	\( T_i(0, x) \) are trivial, so
	\( x = \inv{x} \) for
	\( x \in e_{- i} R e_i \),
	\( i \neq 0 \). Existence of remaining basic Weyl elements (both for
	\( r < d \) and for
	\( r = d \)) implies that there are
	\( e_{\mp 1, \pm 1} \in e_{\mp 1} R e_{\pm 1} \) such that
	\( e_{\pm 1, \mp 1} e_{\mp 1, \pm 1} = e_1 \). Then the module
	\( M = e_0 R e_1 \) has traditional non-degenerate quadratic form
	\( q(m) \) such that
	\( ( m, - e_{- 1, 1} q(m) ) \in \Delta \) and
	\( (R, \Delta) \) is constructed by
	\( M \perp \Hyp(K^r) \), so
	\( \Unit(R, \Delta) = \Unit(r, M) \) is an orthogonal group.
	
	For remaining Tits indices there is a basic Weyl element
	\[
		T_1( 0, e_{- 1, 1} )\,
		T_{- 1}( 0, - e_{1, - 1} )\,
		T_1( 0, e_{- 1, 1} )
	\]
	with long root, where necessarily
	\( e_{\pm 1, \mp 1} e_{\mp 1, \pm 1} = e_{\pm 1} \). Since
	\( ( 0, e_{\mp 1, \pm 1} ) \in \Delta \), we have
	\( \inv{ e_{\mp 1, \pm 1} } = - e_{\mp 1, \pm 1} \). The algebra
	\( S = e_1 R e_1 \) has the
	\( K \)-linear involution
	\( x \mapsto x^* = e_{1, - 1} \inv{x} e_{- 1, 1} \) and the form parameter
	\(
		\Lambda
		=
		\{
			x \in S
			\mid
			( 0, e_{- 1, 1} x ) \in \Delta
		\}
	\) with respect to
	\( \lambda = - 1 \). The
	\( S \)-module
	\( M = e_0 R e_1 \) has non-degenerate hermitian form
	\( B(m, m') = e_{1, - 1} \inv{m} m' \) and quadratic form
	\(
		q(m)
		=
		\{
			x \in S
			\mid
			( m, - e_{- 1, 1} x ) \in \Delta
		\}
	\), so the corresponding odd form parameter is
	\(
		\OddFormPar
		=
		\{
			(m, x) \in M \times S
			\mid
			(m, e_{- 1, 1} x) \in \Delta
		\}
	\). It is easy to check that
	\( (R, \Delta) \) is the odd form ring constructed by
	\( M \perp \Hyp(S^r) \), so
	\( \Unit(R, \Delta) = \Unit(r, M) \).
\end{proof}

For simple group schemes with absolute root systems of types
\( \RootSys{A}{\ell} \) and
\( \RootSys{C}{\ell} \) the odd form parameter
\( \Delta \) is the maximal one. It follows that in the isotropic case the odd form parameter
\( \OddFormPar \) is also the maximal one, equivalently,
\( \Lambda = \Lambda_{\max} \).

\section{%
	Solvability of%
	\texorpdfstring{
		\( \KFunc_1 \)
	}{K1}%
}

\begin{lemma}
	\label{lin-ab-k1}
	If
	\( S \) be a semilocal associative unital ring, then
	\[
		\KFunc_1(n, S) \to \KFunc_1(n + 1, S)
	\]
	is surjective for
	\( n \geq 1 \) and injective for
	\( n \geq 2 \). Moreover,
	\[
		\KFunc_1(2, S)
		\cong
		\KFunc_1(3, S)
		\cong
		\ldots
	\]
	is abelian.
\end{lemma}
\begin{proof}
	This is a special case of
	\cite[theorems V.4.1(b) and V.4.2]{k-theory}.
\end{proof}

\begin{lemma}
	\label{uni-ab-k1}
	Let
	\( S \) be a semilocal odd form ring and
	\( (M, B, q) \) a quadratic
	\( S \)-module. Then
	\[
		\KUnit_1(2 \ell, M)
		\to
		\KUnit_1(2 \ell + 2, M)
	\]
	is surjective for
	\( \ell \geq 1 \) and injective for
	\( \ell \geq 2 \). If in addition
	\( M_S \) is free of finite rank and
	\( B \) is non-degenerate, then
	\[
		\KUnit_1(4, M)
		\cong
		\KUnit_1(6, M)
		\cong
		\ldots
	\]
	is abelian.
\end{lemma}
\begin{proof}
	The first claim is a special case of
	\cite[theorem 5]{odd-uni-gro} and
	\cite[theorem 1.1]{bc-inj-sta}. To prove the second claim we only have to check
	\(
		[ \Unit(1, M), \Unit(1, M) ]
		\leq
		\ElemUnit(2 \ell, M)
	\)
	for sufficiently large
	\( \ell \). The idea is to embed
	\( M \) into a hyperbolic space as in
	\cite[5.1.19]{hah-o-mea} and move this space into sufficiently high hyperbolic planes of
	\( M \oplus \Hyp(S^\ell) \) as in Whitehead lemma, see e.g.\ %
	\cite[5.4.3]{hah-o-mea} for the case of even unitary groups.
	
	Consider the quadratic module
	\( M' = M \oplus M \) with the forms
	\begin{align*}
			B( m_1 \oplus m_2, n_1 \oplus n_2 )
			&=
			B(m_1, n_1) - B(m_2, n_2),
		&
			q( m_1 \oplus m_2 )
			&=
			q(m_1) - q(m_2).
	\end{align*}
	This module is hyperbolic. Namely, choose a sesquilinear form
	\( Q \colon M \times M \to S \) generating
	\( B \) and
	\( q \) using lemma
	\ref{qua-eve}. Since
	\( B \) is non-degenerate, there is a linear map
	\( f \colon M \to M \) such that
	\( B(f(m), n) = - Q(m, n) \), so
	\( B(m, f(n)) = - Q(n, m)^* \lambda \). Consider the submodules
	\begin{align*}
			N_1
			&=
			\{ m \oplus m \mid m \in M \},
		&
			N_2
			&=
			\{ (m + f(m)) \oplus f(m) \mid m \in M \}
	\end{align*}
	of
	\( M' \). Clearly, 
	\( M' = N_1 \oplus N_2 \) and the forms on
	\( M' \) are generated by the sesquilinear form
	\[
		Q'\bigl(
				(m_1 \oplus m_1)
				+
				\bigl(
					( m_2 + f(m_2) ) \oplus f(m_2)
				\bigr)
			,
				(m_1' \oplus m_1')
				+
				\bigl(
					( m_2' + f(m_2') ) \oplus f(m_2')
				\bigr)
		\bigr)
		=
		B(m_1, m_2').
	\]
	This sesquilinear form induces an isomorphism
	\( N_1 \cong \Hom_S(N_2, S) \).
	
	It follows that
	\( e_{- 1} S \oplus M \oplus e_1 S \) is a direct summand of
	\(
		\Hyp( S^{\ell - 1} )
		=
		\bigoplus_{i = - \ell}^{- 2}
			e_i S
		\oplus
		\bigoplus_{i = 2}^\ell
			e_i S
	\) for sufficiently large
	\( \ell \geq 2 \) via an embedding
	\(
		u
		\colon
		e_{- 1} S \oplus M \oplus e_1 S
		\to
		\Hyp( S^{ \ell - 1 } )
	\). By Witt cancellation theorem
	\cite[theorem 1]{ove-uni} there is
	\( h \in \Unit(2 \ell, M) \) such that
	\( h|_{ e_{- 1} S \oplus M \oplus e_1 S } = u \). Again applying stability we can assume that
	\( h \in \ElemUnit(2 \ell, M) \). Now if
	\( g_1, g_2 \in \Unit(1, M) \), then
	\[
		[g_1, g_2]
		\equiv
		[ g_1, \up{h}{g_2} ]
		=
		1
		\pmod{ \ElemUnit(2 \ell, M) }.
		\qedhere
	\]
\end{proof}

To deal with the case
\( \TitsIndex{E}{1}{6}{2}{28} \) in theorem
\ref{sol-sem-iso} below we need the following lemma about orthogonal groups of isotropic rank
\( 1 \). Let
\( M \) be a finite projective module of constant even rank over
\( K \) with a traditional non-degenerate quadratic form
\( q \colon M \to K \). Recall that
\( \SpecOrth(q) \leq \Orth(q) \) is the subgroup of elements with trivial Dickson invariant, it is a normal subgroup containing
\( [ \Orth(q), \Orth(q) ] \). A
\textit{reflection} in
\( \Orth(q) \) is an element
\(
	s_v
	\colon
	m
	\mapsto
	m - \frac{ B(m, v) }{ q(v) } v
\), where
\( v \in M \) and
\( q(v) \in K^* \). Let
\( \Refl(q) \leq \Orth(q) \) be the subgroup generated by reflections and
\( \Refl^{+}(q) \leq \Refl(q) \cap \SpecOrth(q) \) the subgroup generated by products of pairs of reflections.

\begin{lemma}
	\label{ort-ab-k1}
	Let
	\( M \neq 0 \) be a free module of finite even rank over a semilocal ring
	\( K \) with a traditional non-degenerate quadratic form
	\( q \colon M \to K \). Then
	\( [ \Orth(q), \Orth(q) ] \leq \Refl^{+}(q) \) and the factor-group
	\( \mathrm{KO}_1(2, q) \) of
	\( \Orth( q \perp \Hyp(K) ) \) by its elementary subgroup is abelian.
\end{lemma}
\begin{proof}
	Firstly suppose that
	\( K \) is a field. By the classical Cartan--Dieudonn\'e theorem
	\( \Orth(q) = \Refl(q) \) and
	\( \SpecOrth(q) = \Refl^{+}(q) \) unless
	\( K \cong \Field_2 \) has two elements and
	\( M \cong \Hyp(K^2) \). In the exceptional case
	\(
		\Orth(q)
		\cong
		( \Sym_3 \times \Sym_3 )
		\rtimes
		\Int / 2 \Int
	\), the six reflections correspond to the elements
	\( ( \sigma, \sigma^{- 1} ) \rtimes [1] \) (the unique conjugacy class in
	\( ( \Sym_3 \times \Sym_3 ) \rtimes [1] \) with six elements), and under this isomorphism
	\begin{align*}
			\SpecOrth(q) &\cong \Sym_3 \times \Sym_3,
		\\
			\Refl(q)
			&\cong
			\{
				(\sigma, \tau) \in \Sym_3 \times \Sym_3
				\mid
				\sigma \tau
				\text{ is even}
			\}
			\rtimes
			\Int / 2 \Int,
		\\
			\Refl^{+}(q)
			&\cong
			\{
				(\sigma, \tau) \in \Sym_3 \times \Sym_3
				\mid
				\sigma \tau
				\text{ is even}
			\}.
	\end{align*}
	So for all fields
	\( [ \Orth(q), \Orth(q) ] \leq \Refl^{+}(q) \).
	
	Now return to the general case. Let
	\( \Jac(K) \leqt K \) be the Jacobson radical, so
	\( K / \Jac(K) = \prod_{i = 1}^n F_i \) is a product of fields. The map
	\[
		\Orth(q)
		\to
		\Orth( q_{ K / \Jac(K) } )
		=
		\prod_{i = 1}^n
			\Orth( q_{F_i} )
	\]
	induces the surjection
	\[
		\Refl^{+}(q)
		\to
		\prod_{i = 1}^n
			\Refl^{+}( q_{F_i} ).
	\]
	By
	\cite[theorem 4.2]{ort-sem} the kernel of
	\( \Orth(q) \to \Orth( q_{ K / \Jac(K) } ) \) is contained in
	\( \Refl^{+}(q) \), so the first claim follows from the field case.
	
	To prove the second claim note that
	\begin{align*}
			T_{+}(u)
			&=
			\Bigl(
				\begin{smallmatrix}
						1 & - B(u, {-}) & - q(u)
					\\
						0 & 1 & u
					\\
						0 & 0 & 1
				\end{smallmatrix}
			\Bigr),
		&
			T_{-}(u)
			&=
			\Bigl(
				\begin{smallmatrix}
						1 & 0 & 0
					\\
						u & 1 & 0
					\\
						- q(u) & - B(u, {-}) & 1
				\end{smallmatrix}
			\Bigr),
	\end{align*}
	for
	\( u \in M \) lie in the elementary subgroup of
	\( e_{- 1} K \oplus M \oplus e_1 K \) and
	\[
		T_{+}(v)\, T_{-}( v / q(v) )\, T_{+}(v)
		=
		\Bigl(
			\begin{smallmatrix}
					0 & 0 & q(u)
				\\
					0 & S_v & 0
				\\
					- 1 / q(u) & 0 & 0
			\end{smallmatrix}
		\Bigr).
	\]
	By the Gauss decomposition
	\( \Orth( q \perp \Hyp(K) ) \) is generated by
	\( T_{\pm}(u) \) and diagonal matrices (here we use that
	\( M \neq 0 \)), so the result follows from the first claim.
\end{proof}

Another non-trivial case in theorem
\ref{sol-sem-iso} below is
\( \TitsIndex{E}{}{7}{2}{31} \), this requires even more preliminary results. We need two classes of form algebras over
\( K \). Let us call a form ring
\( S \) a
\textit{%
	quadratic form
	\( K \)-algebra%
} if
\( S \) is a quadratic \'etale
\( K \)-algebra with the standard involution (the only geometrically non-trivial
\( K \)-linear involution),
\( \lambda = - 1 \), and
\( \Lambda = \Lambda_{\min} = \Lambda_{\max} = K \). In other words,
\( S \) is a twisted form of
\( K \times K \) with
\begin{align*}
		(x, y)^* &= (y, x),
	&
		\lambda &= - 1,
	&
		\Lambda &= \{ (x, x) \mid x \in K \}.
\end{align*}
Similarly, a form ring
\( S \) is called a
\textit{%
	quaternion form
	\( K \)-algebra%
} if
\( S \) is a quaternion
\( K \)-algebra with the standard involution
\( x^* = \mathrm{tr}(x) - x \),
\( \lambda = - 1 \), and
\( \Lambda = \Lambda_{\min} = K \). Such form algebras are precisely twisted forms of
\( \Mat(2, K) \) with
\begin{align*}
		\sMat{x}{y}{z}{w}^* &= \sMat{w}{- y}{- z}{x},
	&
		\lambda &= - 1,
	&
		\Lambda
		&=
		\{ \sMat{x}{0}{0}{x} \mid x \in K \}.
\end{align*}

\begin{lemma}
	\label{her-dia}
	Let
	\( S \) be a quadratic or quaternion form algebra over semilocal
	\( K \) and
	\( (M, B, q) \) a quadratic module over
	\( S \). Suppose that
	\( M \) is projective of constant finite rank
	\( n \) over
	\( S \) (i.e.\ of rank
	\( 2 n \) or
	\( 4 n \) over
	\( K \) respectively) and
	\( B \) is non-degenerate. Then
	\( M \) has an orthogonal basis
	\( e_1, \ldots, e_n \in M \), i.e.\ %
	\( B(e_i, e_j) = 0 \) for
	\( i \neq j \).
\end{lemma}
\begin{proof}
	In both cases
	\( S \) is a composition
	\( K \)-algebra with the standard involution,
	\( \lambda = - 1 \), and
	\( \Lambda = K \). Without loss of generality
	\( K \) is a field, so
	\( M \) has some basis
	\( e_1, \ldots, e_n \in M \). Arguing by induction on
	\( n \geq 1 \) it suffices to prove that there exist
	\( s_2, \ldots, s_n \in S \) such that
	\[
		B\Bigl(
			s_1 + \sum_{i = 2}^n e_i s_i,
			s_1 + \sum_{i = 2}^n e_i s_i
		\Bigr)
		\neq
		0.
	\]
	
	Assume the contrary. If we chose a basis of
	\( S \), then the above expression becomes a set of polynomials in coordinates of
	\( s_i \) in some basis of
	\( S \) over
	\( K \). If all of these polynomials are zero, then the hermitian form is identically zero contradicting non-degeneracy. Thus some of them is a non-zero polynomial of total degree
	\( 2 \) taking only zero values. It follows that
	\( K \cong \Field_2 \) has two elements.
	
	We have
	\begin{align*}
			B(e_1, e_1) &= 0,
		\\
			s^* B(e_i, e_j) t &= t^* B(e_i, e_j)^* s
			\text{ for }
			i > j > 1,
			\tag{%
				\( * \)%
			}
		\\
			s^* B(e_i, e_i) s
			&=
			B(e_1, e_i) s + s^* B(e_1, e_i)^*
			\text{ for }
			i > 1.
			\tag{%
				\( * * \)%
			}
	\end{align*}
	Taking
	\( s = t = 1 \) in (%
		\( * \)%
	) we get
	\( B(e_i, e_j)^* = B(e_i, e_j) \), and next taking only
	\( t = 1 \) we obtain
	\( B(e_i, e_j) s = s^* B(e_i, e_j) \). In the quaternion case this means that
	\(
		B(e_i, e_j) s t
		=
		t^* s^* B(e_i, e_j)
		=
		B(e_i, e_j) t s
	\), but additive commutators
	\( [t, s] \) generate
	\( S \) as a right ideal, so
	\( B(e_i, e_j) = 0 \). In the quadratic case the trace map
	\( S \to K,\, s \mapsto s + s^* \) is surjective and again
	\( B(e_i, e_j) = 0 \).
	
	Now consider (%
		\( * * \)%
	). Linearizing this identity we get
	\( s^* B(e_i, e_i) t = t^* B(e_i, e_i) s \), so
	\( B(e_i, e_i) = 0 \) by the above argument. But then (%
		\( * * \)%
	) reduces to
	\( B(e_1, e_i) s + s^* B(e_1, e_i)^* \) and again
	\( B(e_1, e_i) = 0 \). So
	\( B \) is identically zero, a contradiction.
\end{proof}

The next result is proved in a more general form in
\cite[theorem 4.6]{uni-sem}, but assuming that
\( S \) has no residue fields with
\( 2 \) elements. We give a simpler proof only for quadratic form algebras. Recall that if
\( M \) is a quadratic module over a quadratic form algebra
\( S \) such that
\( M_S \) is free of finite rank and the hermitian form is non-degenerate, then
\( \SpecUnit(M) = \Unit(M) \cap \SpecLin(M) \), i.e.\ the group of unitary matrices with the determinant
\( 1 \in S \). As for unitary and elementary groups,
\( \SpecUnit(n, M) = \SpecUnit( M \perp \Hyp(S^n) ) \).

\begin{lemma}
	\label{2a-ab-k1}
	Let
	\( S \) be a quadratic form algebra over semilocal
	\( K \) and
	\( M \) a quadratic module over
	\( S \). Suppose that the hermitian form is non-degenerate and the rank of
	\( M \) over
	\( K \) is constant. Then
	\( \ElemUnit(1, M) = \SpecUnit(1, M) \) and
	\( \KUnit_1(1, M) \) is abelian.
\end{lemma}
\begin{proof}
	By lemma
	\ref{her-dia} the module
	\( M \) has an orthogonal basis
	\( e_1, \ldots, e_n \in M \). We prove that arbitrary
	\( g \in \Unit(1, M) \) can be multiplied by elementary transformations from the right to get an isometry stabilizing all
	\( e_i \). Recall that root elements have the form
	\begin{align*}
			T_{+}(m, x)
			&=
			\Bigl(
				\begin{smallmatrix}
						1 & B( m, {-} ) & x
					\\
						0 & 1 & m
					\\
						0 & 0 & 1
				\end{smallmatrix}
			\Bigr),
		&
			T_{-}(m, x)
			&=
			\Bigl(
				\begin{smallmatrix}
						1 & 0 & 0
					\\
						m & 1 & 0
					\\
						- x & - B( m, {-} ) & 1
				\end{smallmatrix}
			\Bigr)
	\end{align*}
	as endomorphisms of
	\( e_{-} S \oplus M \oplus e_{+} S \), where
	\( m \in M \),
	\( x \in S \), and
	\( q(m) = x + K \).
	
	Assume that
	\( g e_i = e_i \) for
	\( i < k \), so
	\( B(g e_k, e_i) = 0 \) for
	\( i < k \). Let
	\( g e_k = e_{-} a \oplus m \oplus e_{+} b \). If
	\( a \in S^* \), then
	\[
		T_{+}(u, q'(u))\,
		T_{-}\bigl(
			(e_k - m) a^{- 1},
			B(m - e_k, m) (a a^*)^{- 1} + b a^{- 1}
		\bigr)\,
		g
	\]
	stabilizes
	\( e_i \) for
	\( i \leq k \), where
	\( u \in M \) is a vector such that
	\( B(u, e_k) = - a \) and
	\( q'(u) \) is a lift of
	\( q(u) \) to
	\( S \). So the claim about
	\( g \) follows by induction.
	
	Now suppose that
	\( g \in \SpecUnit(1, M) \) stabilizes
	\( M \) pointwise. The group of such elements is isomorphic to
	\( \SpecUnit( \Hyp(S) ) \cong \SpecLin(2, K) \) and this isomorphic preserves both root subgroups. The result from the statement now follows from
	\( \SpecLin(2, K) = \Elem(2, K) \).
\end{proof}

\begin{theorem}
	\label{sol-sem-iso}
	Let
	\( G \) be a simple group scheme over a semilocal ring
	\( K \) of isotropic rank at least
	\( 2 \). If the Tits index is
	\( \TitsIndex{E}{2}{6}{2}{16''} \) or
	\( \TitsIndex{E}{}{8}{2}{78} \) assume that
	\( K \) contains a field. Then the group
	\( \KFunc_1^G(K) \) is solvable with derived length bounded by an absolute constant.
\end{theorem}
\begin{proof}
	We freely use lemma
	\ref{sol-iso} to replace
	\( G \) by an isogenic group scheme if necessary. Classical Tits indices are covered by lemmas
	\ref{lin-ab-k1} and
	\ref{uni-ab-k1} using theorem
	\ref{twi-iso-cla}. By the Gauss decomposition
	\( G(K) \) is generated by
	\( \Elem_G(K) \) and
	\( L(K) \), so the claim trivially holds if
	\( L \) is commutative. This holds if
	\( G \) is quasi-split, i.e.\ for the Tits indices
	\[
		\TitsIndex{D}{s}{4}{2}{2},\,
		\TitsIndex{E}{1}{6}{6}{0},\,
		\TitsIndex{E}{2}{6}{4}{2},\,
		\TitsIndex{E}{}{7}{7}{0},\,
		\TitsIndex{E}{}{8}{8}{0},\,
		\TitsIndex{F}{}{4}{4}{0},\,
		\TitsIndex{G}{}{2}{2}{0}.
	\]
	The cases
	\( \TitsIndex{E}{2}{6}{2}{16''} \) and
	\( \TitsIndex{E}{}{8}{2}{78} \) are proven in
	\cite[corollary 6.11(10)]{k1-red-tri}.
	
	\begin{figure}[ht]
		
		\centering
		
		\begin{tikzpicture}[
			node/.style = {
				circle,
				draw,
				inner sep = 0,
				minimum size = 6
			},
			chosen/.style = {
				circle,
				draw,
				double,
				inner sep = 0,
				minimum size = 6
			},
			affine/.style = {
				circle,
				draw,
				fill = black!50,
				double,
				inner sep = 0,
				minimum size = 6
			},
		]
			
			\newcommand{\edgelen}{0.8}
			\newcommand{\xshift}{8cm}
			\newcommand{\yshift}{4cm}
			
			\begin{scope}[
				scale = \edgelen,
				xshift = { 0 * \xshift },
				yshift = { 0 * \yshift },
			]
				
				\node at (0, 0) {
					\(
						\TitsIndex{D}{1}{5}{1}{ (1) }
						\subseteq
						\TitsIndex{E}{1}{6}{2}{28}
					\)
				};
				
				\node[chosen] (v1-1) at (- 2, 1) { };
				\node[node] (v1-2) at (0, 2) { };
				\node[node] (v1-3) at (- 1, 1) { };
				\node[node] (v1-4) at (0, 1) { };
				\node[node] (v1-5) at (1, 1) { };
				\node[chosen] (v1-6) at (2, 1) { };
				
				\draw (v1-4) -- (v1-3) -- (v1-1);
				\draw (v1-4) -- (v1-5) -- (v1-6);
				\draw (v1-4) -- (v1-2);
				
				\draw[dashed, rounded corners]
					(- 2.5, 0.5)
					rectangle
					(1.5, 2.5);
				
			\end{scope}
			
			\begin{scope}[
				scale = \edgelen,
				xshift = { 1 * \xshift },
				yshift = { 0 * \yshift },
			]
				
				\node at (0, 0) {
					\(
						\TitsIndex{A}{1}{5}{1}{ (3) }
						\subseteq
						\TitsIndex{E}{1}{6}{2}{16}
					\)
				};
				
				\node[node] (v2-1) at (- 2, 1) { };
				\node[ chosen, label = right: { l } ]
					(v2-2) at (0, 2) { };
				\node[node] (v2-3) at (- 1, 1) { };
				\node[ chosen, label = 10: { s } ]
					(v2-4) at (0, 1) { };
				\node[node] (v2-5) at (1, 1) { };
				\node[node] (v2-6) at (2, 1) { };
				
				\draw (v2-4) -- (v2-3) -- (v2-1);
				\draw (v2-4) -- (v2-5) -- (v2-6);
				\draw (v2-4) -- (v2-2);
				
				\draw[dashed, rounded corners]
					(- 2.5, 0.5)
					rectangle
					(2.5, 1.5);
				
			\end{scope}
			
			\begin{scope}[
				scale = \edgelen,
				xshift = { 0 * \xshift },
				yshift = { - 1 * \yshift },
			]
				
				\node at (0, 0) {
					\(
						\TitsIndex{D}{2}{5}{2}{ (1) }
						\subseteq
						\TitsIndex{E}{2}{6}{2}{16'}
					\)
				};
				
				\node[ affine, label = above: { l } ]
					(v3-0) at (- 2, 1.5) { };
				\node[ chosen, label = above: { us } ]
					(v3-1) at (2, 2) { };
				\node[ chosen, label = above: { s } ]
					(v3-2) at (- 1, 1.5) { };
				\node[node] (v3-3) at (1, 2) { };
				\node[node] (v3-4) at (0, 1.5) { };
				\node[node] (v3-5) at (1, 1) { };
				\node[ chosen, label = below: { us } ]
					(v3-6) at (2, 1) { };
				
				\draw
					(v3-4)
					to[ out = 45, in = 180 ]
					(v3-3);
				\draw (v3-3) -- (v3-1);
				\draw
					(v3-4)
					to[ out = - 45, in = 180 ]
					(v3-5);
				\draw (v3-5) -- (v3-6);
				\draw (v3-4) -- (v3-2) -- (v3-0);
				
				\draw[dotted] (v3-1) -- (v3-6);
				\draw[dotted] (v3-3) -- (v3-5);
				
				\draw[dashed, rounded corners]
					(- 2.5, 0.5)
					rectangle
					(1.5, 2.5);
				
			\end{scope}
			
			\begin{scope}[
				scale = \edgelen,
				xshift = { 1 * \xshift },
				yshift = { - 1 * \yshift },
			]
				
				\node at (0, 0) {
					\(
						\TitsIndex{D}{1}{6}{2}{ (1) },
						\TitsIndex{D}{1}{6}{1}{ (2) }
						\subseteq
						\TitsIndex{E}{}{7}{2}{31}
					\)
				};
				
				\node[ affine, label = above: { l } ]
					(v4-0) at (- 3, 1) { };
				\node[ chosen, label = above: { s } ]
					(v4-1) at (- 2, 1) { };
				\node[node] (v4-2) at (0, 2) { };
				\node[node] (v4-3) at (- 1, 1) { };
				\node[node] (v4-4) at (0, 1) { };
				\node[node] (v4-5) at (1, 1) { };
				\node[ chosen, label = above: { us } ]
					(v4-6) at (2, 1) { };
				\node[node] (v4-7) at (3, 1) { };
				
				\draw (v4-4) -- (v4-3);
				\draw (v4-3) -- (v4-1);
				\draw (v4-1) -- (v4-0);
				\draw (v4-4) -- (v4-5);
				\draw (v4-5) -- (v4-6);
				\draw (v4-6) -- (v4-7);
				\draw (v4-4) -- (v4-2);
				
				\draw[dashed, rounded corners]
					(- 3.5, 0.5)
					rectangle
					(1.5, 2.4);
				\draw[dashed, rounded corners]
					(- 1.5, 0.6)
					rectangle
					(3.5, 2.5);
				
			\end{scope}
			
			\begin{scope}[
				scale = \edgelen,
				xshift = { 0 * \xshift },
				yshift = { - 2 * \yshift },
			]
				
				\node at (0, 0) {
					\(
						\TitsIndex{D}{1}{8}{2}{ (1) }
						\subseteq
						\TitsIndex{E}{}{8}{2}{66}
					\)
				};
				
				\node[ affine, label = above: { l } ]
					(v5-0) at (3.5, 1) { };
				\node[ chosen, label = above: { us } ]
					(v5-1) at (- 3.5, 1) { };
				\node[node] (v5-2) at (- 1.5, 2) { };
				\node[node] (v5-3) at (- 2.5, 1) { };
				\node[node] (v5-4) at (- 1.5, 1) { };
				\node[node] (v5-5) at (- 0.5, 1) { };
				\node[node] (v5-6) at (0.5, 1) { };
				\node[node] (v5-7) at (1.5, 1) { };
				\node[ chosen, label = above: { s } ]
					(v5-8) at (2.5, 1) { };
				
				\draw (v5-4) -- (v5-3) -- (v5-1);
				\draw (v5-4) -- (v5-2);
				\draw (v5-4) -- (v5-5);
				\draw (v5-5) -- (v5-6);
				\draw (v5-6) -- (v5-7);
				\draw (v5-7) -- (v5-8);
				\draw (v5-8) -- (v5-0);
				
				\draw[dashed, rounded corners]
					(- 3, 0.5)
					rectangle
					(4, 2.5);
				
			\end{scope}
			
			\begin{scope}[
				scale = \edgelen,
				xshift = { 1 * \xshift },
				yshift = { - 2 * \yshift },
			]
				
				\node at (0, 0) {
					\(
						\TitsIndex{D}{1}{5}{1}{ (1) }
						\subseteq
						\TitsIndex{E}{}{7}{3}{28}
					\)
				};
				
				\node[ chosen, label = above: { s } ]
					(v6-1) at (- 2.5, 1) { };
				\node[node] (v6-2) at (- 0.5, 2) { };
				\node[node] (v6-3) at (- 1.5, 1) { };
				\node[node] (v6-4) at (- 0.5, 1) { };
				\node[node] (v6-5) at (0.5, 1) { };
				\node[ chosen, label = above: { s } ]
					(v6-6) at (1.5, 1) { };
				\node[ chosen, label = above: { l } ]
					(v6-7) at (2.5, 1) { };
				
				\draw (v6-4) -- (v6-3) -- (v6-1);
				\draw (v6-4) -- (v6-2);
				\draw (v6-4) -- (v6-5);
				\draw (v6-5) -- (v6-6);
				\draw (v6-6) -- (v6-7);
				
				\draw[dashed, rounded corners]
					(- 3, 0.5)
					rectangle
					(1, 2.5);
				
			\end{scope}
			
			\begin{scope}[
				scale = \edgelen,
				xshift = { 0 * \xshift },
				yshift = { - 3 * \yshift },
			]
				
				\node at (0, 0) {
					\(
						\TitsIndex{D}{1}{6}{3}{ (2) }
						\subseteq
						\TitsIndex{E}{}{7}{4}{9}
					\)
				};
				
				\node[ chosen, label = above: { l } ]
					(v7-1) at (- 2.5, 1) { };
				\node[node] (v7-2) at (- 0.5, 2) { };
				\node[ chosen, label = above: { l } ]
					(v7-3) at (- 1.5, 1) { };
				\node[ chosen, label = 70: { s } ]
					(v7-4) at (- 0.5, 1) { };
				\node[node] (v7-5) at (0.5, 1) { };
				\node[ chosen, label = above: { s } ]
					(v7-6) at (1.5, 1) { };
				\node[node] (v7-7) at (2.5, 1) { };
				
				\draw (v7-4) -- (v7-3) -- (v7-1);
				\draw (v7-4) -- (v7-2);
				\draw (v7-4) -- (v7-5);
				\draw (v7-5) -- (v7-6);
				\draw (v7-6) -- (v7-7);
				
				\draw[dashed, rounded corners]
					(- 2, 0.5)
					rectangle
					(3, 2.5);
				
			\end{scope}
			
			\begin{scope}[
				scale = \edgelen,
				xshift = { 1 * \xshift },
				yshift = { - 3 * \yshift },
			]
				
				\node at (0, 0) {
					\(
						\TitsIndex{D}{1}{5}{1}{ (1) }
						\subseteq
						\TitsIndex{E}{}{8}{4}{28}
					\)
				};
				
				\node[ chosen, label = above: { s } ]
					(v8-1) at (- 3, 1) { };
				\node[node] (v8-2) at (- 1, 2) { };
				\node[node] (v8-3) at (- 2, 1) { };
				\node[node] (v8-4) at (- 1, 1) { };
				\node[node] (v8-5) at (0, 1) { };
				\node[ chosen, label = above: { s } ]
					(v8-6) at (1, 1) { };
				\node[ chosen, label = above: { l } ]
					(v8-7) at (2, 1) { };
				\node[ chosen, label = above: { l } ]
					(v8-8) at (3, 1) { };
				
				\draw (v8-4) -- (v8-3) -- (v8-1);
				\draw (v8-4) -- (v8-2);
				\draw (v8-4) -- (v8-5);
				\draw (v8-5) -- (v8-6);
				\draw (v8-6) -- (v8-7);
				\draw (v8-7) -- (v8-8);
				
				\draw[dashed, rounded corners]
					(- 3.5, 0.5)
					rectangle
					(0.5, 2.5);
				
			\end{scope}
			
		\end{tikzpicture}
		
		\caption{%
			Exceptional Tits indices and their subindices from the proof of theorem
			\ref{sol-sem-iso}. The highest roots in affine diagrams are denoted by gray circles. Letters correspond to the lengths of relative roots.%
		}
		
	\end{figure}
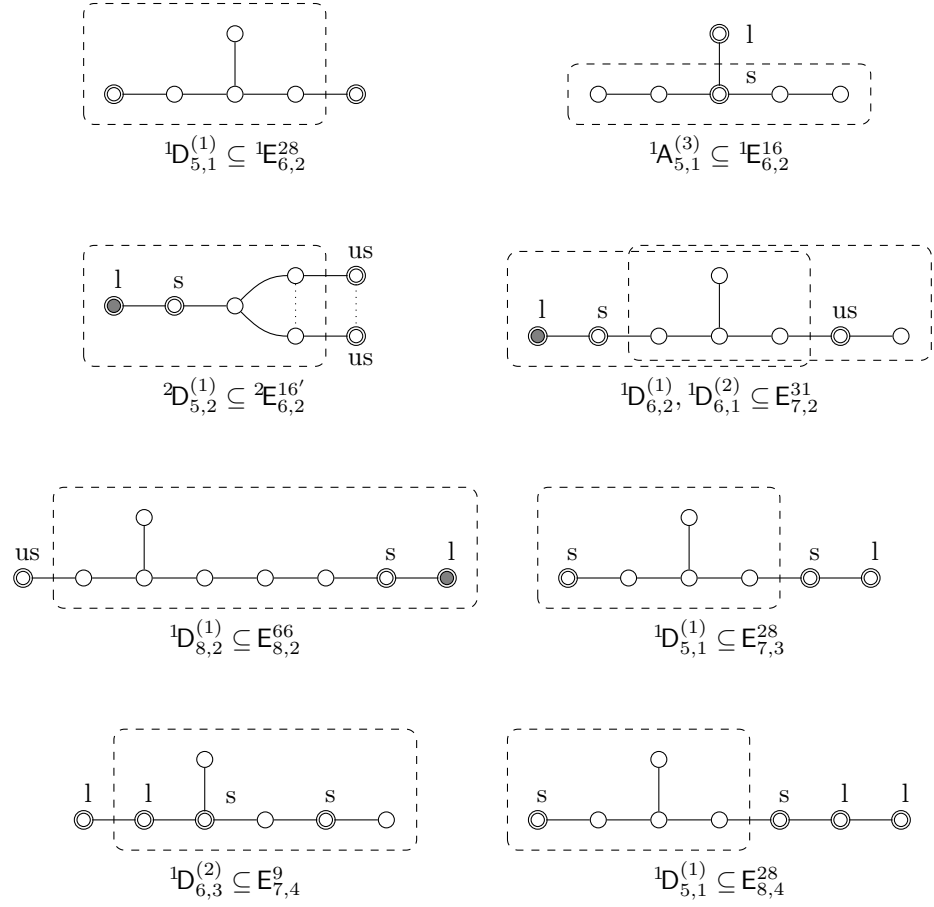
	
	In the case
	\( \TitsIndex{E}{1}{6}{2}{16} \) the relative root system has type
	\( \RootSys{G}{2} \) and the absolute root system of
	\( L \) has type
	\( 2 \RootSys{A}{2} \). Let
	\( H \leq G \) be the closed reductive group subscheme generated by
	\( U_{- \alpha} \),
	\( U_\alpha \) as an fppf sheaf for a short root
	\( \alpha \). Considering the root diagram of
	\( \RootSys{E}{6} \) from
	\cite{atlas} one can see that
	\( H \) is simple with the Tits index
	\( \TitsIndex{A}{1}{5}{1}{ (3) } \) and contains
	\( [L, L] \). We are done by theorem
	\ref{twi-iso-cla} and lemma
	\ref{lin-ab-k1}.
	
	Next consider
	\( \TitsIndex{E}{2}{6}{2}{16'} \) and
	\( \TitsIndex{E}{}{8}{2}{66} \). In both these cases the relative root system has type
	\( \RootSys{BC}{2} \) and the scheme derived subgroup
	\( [L, L] \) is simple. Let
	\( H \leq G \) be the closed reductive group subscheme generated by root subgroups with roots from
	\( \RootSys{C}{2} \subseteq \RootSys{BC}{2} \), it contains
	\( [L, L] \). The root diagrams from
	\cite{atlas} show that the Tits indices of
	\( H \) are
	\( \TitsIndex{D}{2}{5}{2}{ (1) } \) and
	\( \TitsIndex{D}{1}{8}{2}{ (1) } \), so the result follows from theorem
	\ref{twi-iso-cla} and lemma
	\ref{uni-ab-k1}.
	
	For
	\( \TitsIndex{E}{}{7}{4}{9} \) the relative root system has type
	\( \RootSys{F}{4} \) and the absolute root system of
	\( L \) has type
	\( 3 \RootSys{A}{1} \). Let
	\( H \leq G \) be the closed reductive group subscheme generated by root subgroups with roots from
	\( \RootSys{C}{3} \subseteq \RootSys{F}{4} \). Looking and the root diagram
	\cite{atlas} one can see that the Tits index of
	\( H \) is
	\( \TitsIndex{D}{1}{6}{3}{ (2) } \) and
	\( [L, L] \leq H \), so
	\( \KFunc_1^G(K) \) is again solvable by theorem
	\ref{twi-iso-cla} and lemma
	\ref{uni-ab-k1}.
	
	Now consider the Tits indices
	\( \TitsIndex{E}{1}{6}{2}{28} \),
	\( \TitsIndex{E}{}{7}{3}{28} \), and
	\( \TitsIndex{E}{}{8}{4}{28} \). In these cases
	\( [L, L] \) is simple with the absolute root system of type
	\( \RootSys{D}{4} \). The relative root system of
	\( G \) has types
	\( \RootSys{A}{2} \),
	\( \RootSys{C}{3} \), and
	\( \RootSys{F}{4} \) respectively. Let
	\( H \leq G \) be the closed reductive group subscheme generated by
	\( U_{- \alpha} \) and
	\( U_\alpha \) as an fppf sheaf (where
	\( \alpha \) is short in the last two cases), it contains
	\( [L, L] \). By
	\cite{atlas} this group scheme is simple with the Tits index
	\( \TitsIndex{D}{1}{5}{1}{ (1) } \) and now we apply theorem
	\ref{twi-iso-cla} and lemma
	\ref{ort-ab-k1}.
	
	In the remaining case
	\( \TitsIndex{E}{}{7}{2}{31} \) the relative root system has type
	\( \RootSys{BC}{2} \) and the absolute root system of
	\( L \) has type
	\( \RootSys{A}{1} + \RootSys{D}{4} \). Firstly let
	\( H_1 \leq G \) be the closed reductive group subscheme generated by root subgroups with roots from the relative root subsystem of type
	\( \RootSys{C}{2} \). By
	\cite{atlas} it has the Tits index
	\( \TitsIndex{D}{1}{6}{2}{ (1) } \), so theorem
	\ref{twi-iso-cla} and lemma
	\ref{uni-ab-k1} imply that
	\( (L \cap H_1)(K) \) is solvable modulo its intersection with
	\( \Elem_G(K) \). The group scheme
	\( [L, L] \) is a central product of two simple group subschemes
	\( L_1 \) and
	\( L_2 \), the first one with the absolute root system
	\( \RootSys{D}{4} \) and the second one with the absolute root system
	\( \RootSys{A}{1} \), and
	\( L \cap H_1 \) contains only
	\( L_1 \).
	
	To deal with the second factor consider another closed reductive group subscheme
	\( H_2 \leq G \) generated by root subgroups with roots from a root subsystem of type
	\( \RootSys{BC}{1} \). By
	\cite{atlas} it has the Tits index
	\( \TitsIndex{D}{1}{6}{1}{ (2) } \) and it contains the whole group subscheme
	\( [L, L] \). It remains to prove that
	\( H_2(K) / ( \Elem_{H_2}(K)\, L_1 ) \) is solvable. Again by theorem
	\ref{twi-iso-cla} the group scheme
	\( H_2 \) up to isogeny is the connected component of the unitary group scheme of a quadratic module
	\( M \perp \Hyp(S) \) over a quaternion form
	\( K \)-algebra
	\( S \), where the hermitian form is non-degenerate and
	\( M \) has constant rank
	\( 16 \) over
	\( K \). By lemma
	\ref{her-dia} this module
	\( M \) has an orthogonal basis
	\( (e_1, e_2, e_3, e_4) \). Let
	\( H_3 \leq H_2 \) be the closed reductive group subscheme corresponding to the connected component of the unitary group scheme of the quadratic module
	\( e_1 S \perp \Hyp(S) \). It still contains the factor
	\( L_2 \), but its Tits index is
	\( \TitsIndex{D}{2}{3}{1}{ (2) } \).
	
	This Tits index is isomorphic to
	\( \TitsIndex{A}{2}{3}{1}{ (1) } \). In other words, up to isogeny
	\( H_3 \) is also the special unitary group scheme of a quadratic module
	\( M' \perp \Hyp(S') \) over a quadratic form
	\( K \)-algebra
	\( S' \) by the final application of theorem
	\ref{twi-iso-cla}, where the hermitian form is again non-degenerate and
	\( M' \) has constant rank
	\( 4 \) over
	\( K \). The result follows from lemma
	\ref{2a-ab-k1}.
\end{proof}

\bibliographystyle{plain}
\bibliography{references}

\end{document}